\documentclass[12pt]{article}
\usepackage{amsfonts}
\usepackage{amssymb, amscd, amsthm}
\usepackage{latexsym}
\usepackage{multirow}
\usepackage{amsmath}
\usepackage[all]{xy}
\usepackage[dvips]{graphicx}
\usepackage{mathrsfs}
\usepackage{fancyhdr}
\usepackage{extarrows}

\newtheorem{lemma}[subsection]{Lemma}
\newtheorem{thm}[subsection]{Theorem}
\newtheorem{prop}[subsection]{Proposition}
\newtheorem{rem}[subsection]{Remark}
\newtheorem{coro}[subsection]{Corollary}

\newcommand{\ra}{\rightarrow}
\newcommand{\mo}{\mathcal{O}}
\newcommand{\mf}{\mathcal{F}}
\newcommand{\cd}{\mathcal{D}}
\newcommand{\mg}{\mathcal{G}}
\newcommand{\ma}{\mathcal{A}}
\newcommand{\mb}{\mathcal{B}}
\newcommand{\me}{\mathcal{E}}
\newcommand{\mi}{\mathcal{I}}

\newcommand{\mt}{\mathcal{T}}
\newcommand{\mh}{\mathcal{H}}
\newcommand{\mk}{\mathcal{K}}
\newcommand{\mm}{\mathcal{M}}

\newcommand{\mv}{\mathcal{V}}

\newcommand{\tb}{\mathtt{B}}

\newcommand{\mr}{\mathcal{R}}
\newcommand{\mq}{\mathcal{Q}}

\newcommand{\rd}{\mathscr{D}}

\newcommand{\z}{\Theta}
\newcommand{\dzi}{D_{\Theta_d}^{int}}
\newcommand{\dzo}{D_{\Theta_d}^{o}}
\newcommand{\md}{M(d,0)}
\newcommand{\mdi}{M(d,0)^{int}}
\newcommand{\wrn}{W(r,0,n)}

\newcommand{\lcd}{\lambda_{c^r_n}(d)}
\newcommand{\lrd}{\lambda_r(d)}

\newcommand{\crn}{c^r_n}

\newcommand{\Ext}{\operatorname{Ext}}
\newcommand{\Pic}{\operatorname{Pic}}

\newcommand{\Tor}{\operatorname{Tor}}

\newcommand{\ls}{|dH|}
\newcommand{\p}{\mathbb{P}}
\newcommand{\bc}{\mathbb{C}}

\begin{document}
\fontsize{12pt}{14pt} \textwidth=14cm \textheight=21 cm
\numberwithin{equation}{section}
\title{Moduli spaces of 1-dimensional semi-stable sheaves and Strange duality on $\mathbb{P}^2$.}
\author{Yao YUAN}
\date{}
\maketitle
\begin{flushleft}{\textbf{Abstract.}} We study Le Potier's strange duality conjecture on $\mathbb{P}^2$.   We show the conjecture is true for the pair ($W(2,0,2),~M(d,0)$) with $d>0$,  where $W(2,0,2)$ is the moduli space of semistable sheaves of rank 2, zero first Chern class and second Chern class 2,  and $M(d,0)$ is the moduli space of 1-dimensional semistable sheaves of first Chern class $dH$ and Euler characteristic 0.     

\end{flushleft}

\section{Introduction.}
The strange duality conjecture is a very interesting and famous problem in the theory of moduli spaces of sheaves.  It was first formulated for moduli spaces of vector bundles on curves by Beauville, Donagi and Tu in 1990s (\cite{Bea},\cite{DT}).  Two groups of people proved this conjecture around 2007 (\cite{pbo},\cite{pbt} and \cite{moone}).  Those are very remarkable works.    

So far, there is no general extension of strange duality conjecture to moduli spaces of sheaves over surfaces.  But under some conditions, this conjecture can be formulated, such as Le Potier's formulation on the projective plan (see \cite{nila}) and Marian-Oprea's formulation for K3 and Abelian surfaces (see \cite{motwo}).  In this article, we study the former. 

Let us briefly review the set-up for strange duality conjecture.  More details can be found in \cite{nila} and \cite{motwo}.  

Let $X$ be any smooth projective scheme.  Let $u$ and $c$ be two elements in the Grothendieck group $K(X)$ of coherent sheaves on $X$, assume moreover $u$ is orthogonal to $c$ with respect to the Euler characteristic, i.e. the flat tensor $F_u\otimes^{L} F_c$ is of Euler characteristic zero for any $F_u$ ($F_c$, resp.) a sheaf in class $u$ ($c$, resp.).  Denote by $M_u$ ($M_c$, resp.) the moduli space of semistable sheaves of class $u$ ($c$, resp.), then there is a well-defined determinant line bundle $\lambda_u(c)$ ($\lambda_c(u)$, resp.) associated to $c$ ($u$, resp.) on $M_u$ ($M_c$, resp.).  Actually if there are strictly semistable sheaves, we will require a slightly stronger condition to define $\lambda_c(u)$ and $\lambda_u(c)$.  We refer to Section 2 in \cite{yuan} or Chapter 8 in \cite{hl} for the explicit definition of this ``determinant line bundle".  Notice that the definition in \cite{hl} is dual to what we use in this paper.

The locus $\mathscr{D}:=\{(F_u,F_c)\in M_u\times M_c|H^0(F_u\otimes F_c)\neq 0\}$ is closed in $M_u\times M_c$.  If $\mathscr{D}$ is a divisor of the line bundle $\lambda_u(c)\boxtimes\lambda_c(u)$ (not always the case on surfaces), then the section induced by $\mathscr{D}$ defines the following \emph{strange duality map} up to scalars.
\[SD_{c,u}:H^0(M_c,\lambda_c(u))^{\vee}\ra H^0(M_u,\lambda_u(c)).\]
Strange duality conjecture (for the pair $(u,c)$) says $SD_{c,u}$ is an isomorphism.

In the Le Potier's version of strange duality, $X=\p^2$ with the hyperplane class $H$ and $x\in X$ a single point, $u=u_d:=[\mo_X]-[\mo_X(-dH)]+\frac{d(d-3)}2[\mo_x]$ which is the class of 1-dimensional sheaves with first Chern class $dH$ and Euler characteristic 0, and $c=c^r_n:=r[\mo_{\p^2}]-n[\mo_x]$ which is the class of rank $r$ sheaves of first Chern class zero and second Chern class $n$.  We denote by $[F]$ the class of a sheaf $F$ in $K(X)$. 

Very little is known in general about this conjecture for surfaces, even at numerical level, i.e. whether we have $h^0(M_u,\lambda_u(c))=h^0(M_c,\lambda_c(u))?$  There are some results for special cases, for instance Danila proves that Le Potier's strange duality holds for $(u_d,c^2_n)$ with small $n$ and $d=1,2,3$ (see \cite{nilala} and \cite{nila});  Abe shows that it holds for $(u_d,c^2_n)$ with all $n$ and $d=1,2$ (see \cite{abe});  the author shows that it holds for $(u_d,c^1_n)$ for all $n$ and $d$ (see Section 4.3 in \cite{yuan}); and also a part of work in \cite{GY}.  Marian and Oprea build a version of strange duality and prove that it holds in a large number of cases for generic K3 and abelian surfaces (see \cite{mothree} and \cite{mofour}).

In this article the surface is always $\p^2$ and we study the strange duality for the pair $(u_d,c^2_2),d>0$.  We denote $W(2,0,2):=M_{c^2_2},~\lambda_2(d):=\lambda_{c^2_2}(u_d),~M(d,0):=M_{u_d},~\z^2_d(2):=\lambda_{u_d}(c^2_2)$ and prove the following theorem.
\begin{thm}[Theorem \ref{sd}]\label{begin}The Le Potier's strange duality conjecture is true for the pair $(u_d,c^2_2),d>0$, i.e. the strange duality map
\[SD:=SD_{c^2_2,u_d}:H^0(W(2,0,2),\lambda_2(d))^{\vee}\ra H^0(\md,\z_d^2(2))\]
is an isomorphism.
\end{thm}

The proof of Theorem \ref{begin} is quite tricky: a priori we don't have any numerical evidence of the conjecture at this case, and in fact we still don't know how to compute directly $h^0(M(d,0),\z^2_d(2))$.  But $h^0(W(2,0,2),\lambda_2(d))$ is relatively easy to compute, and we get the injectivity of the map $SD$ by using Fourier transform.  Then we only need to show that $h^0(M(d,0),\z^2_d(2))\leq h^0(W(2,0,2),\lambda_2(d)).$ 

In Section 2 we introduce some notations and also some basic properties of the moduli space $M(d,0)$.  In Section 3 we show the injectivity of $SD$ and finally in Section 4 we show the surjectivity.  At the end there is an appendix where we give more properties of the Fourier transform for future use.


\textbf{Acknowledgements.}  I was supported by NSFC grant 11301292.  I thank E. Looijenga for some helpful discussions.  I also thank the referee for careful reading and helpful comments.

\section{Notations and Preliminaries.}
\begin{enumerate}
\item We are always on $\p^2$ and the base field is $\mathbb{C}$.  $H$ is the hyperplane class in $\p^2$.  $\ls$ is the linear system of the divisor class $dH$.
\item Let $\crn:=r[\mo_{\p^2}]-n[\mo_x]\in K(X)$ with $x$ a single point on $\p^2$ be the class of sheaves with rank $r$, trivial determinant and second Chern class $n$.  
\item Let $u_d:=[\mo_{\p^2}]-[\mo_{\p^2}(-dH)]+\frac{d(d-3)}2[\mo_x],~d>0$ be the class of 1-dimensional sheaves supported at curves in $\ls$ and with Euler characteristic 0.
\item Let $\wrn$ be the moduli space of semi-stable sheaves of class $\crn$, i.e. $\wrn=M_{c^r_n}$.  $\wrn$ is a good quotient of a smooth quasi-projective variety, hence it is normal and Cohen-Macaulay.  $\wrn$ is irreducible (Theorem D in \cite{DLP})  
\item Let $\md$ be the moduli space of semi-stable sheaves of class $u_d$, i.e. $\md=M_{u_d}$.  $\md$ is a good quotient of a smooth quasi-projective variety, hence it is normal and Cohen-Macaulay.  $\md$ is irreducible (Theorem 3.1 in \cite{lee}).  
\item We denote by $\z_d$ the determinant line bundle associated to $[\mo_{\p^2}]$ on $\md$.  Since $dim~H^0(\z_d)=1$ (see \cite{nila} or Theorem 4.3.1 in \cite{yuan}), the line bundle $\z_d$ defines a unique divisor $D_{\z_d}$ which consists of sheaves with non trivial global sections.
   
\item By Proposition 2.8 in \cite{le}, $\z_d^r(n):=\z_d^r\otimes\pi^{*}\mo_{\ls}(n)$ is the determinant line bundle $\lambda_{u_d}(\crn)$ associated to $\crn$ on $\md$ for any $r\geq1$, $n\geq0$, where $\pi:\md\ra\ls$ sends every sheaf to its support.  

\item By Theorem 4.3.1 in \cite{yuan}, $\pi_{*}\z_d\cong\mo_{\ls}$ for any $d>0$.  By Theorem 3.5 in \cite{lee}, for $d=1,2$ $\pi$ is an isomorphism and $\forall r,~\pi_{*}\z_d^r\cong\mo_{\ls}$.  
\item Let $\lcd$ be the determinant line bundle associated to $u_d$ on $\wrn$.  Denote it simply by $\lrd$ if $r=n$.  $\lambda_2(d)=\lambda_u$ in Theorem \ref{begin}.
\end{enumerate}

\section{Injectivity of the strange duality map.}
In this section we prove the following proposition.
\begin{prop}\label{inje}The strange duality map
\begin{equation}\label{sdm}SD:H^0(W(2,0,2),\lambda_2(d))^{\vee}\ra H^0(\md,\z_d^2(2))\end{equation}
is injective for all $d>0$.
\end{prop}

We first recall the Fourier transform on $\p^2$ (see also Section 4 in \cite{lee} or \cite{maru}).  Let $\mathcal{D}$ be the universal curve in $\p^2\times |H|$ as follows.
\begin{equation}\label{original}\xymatrix@C=0.01cm{
  \p^2\times |H|~~~~~\supset&\mathcal{D}\ar[d]^q\ar[r]^p
                & \p^2 \\
               &~~~~~~ |H|\cong\p^2&
               }.
\end{equation}

Let $F$ be a pure 1-dimensional sheaf with Euler characteristic 0, then its Fourier transform is defined to be $G_F:=q_{*}(p^{*}(F\otimes\mo_{\p^2}(2)))\otimes\mo_{|H|}(-1)$.  Let $G$ be a torsion-free sheaf on $|H|$ with first Chern class 0 and Euler characteristic 0, then its Fourier transform is defined to be $F_G:=R^1p_{*}(q^{*}(G\otimes\mo_{|H|}(-1)))\otimes\mo_{\p^2}(-1)$.  We can identify $|H|$ with $\p^2$.  Then these two Fourier transforms in general need not be the inverse to each other, but they provide a birational map as follows.
\begin{equation}\label{bcor}\Phi:\md \dashrightarrow W(d,0,d).\end{equation}

We have the following theorem due to Le Potier (see Lemma 4.2 and Corollary 4.3 in \cite{lee})  
\begin{thm}[Le Potier]\label{lp} $\Phi$ is defined over the complement of the divisor $D_{\z_d}\subset\md$ and induces an isomorphism onto the open subset in $W(d,0,d)$ corresponding to polystable sheaves whose restriction on a generic line $\p^1\cong l\in|H|$ is isomorphic to $\mo_l^{\oplus d}$.  In particular, $\Phi$ is an isomorphism for $d=1,2$.
\end{thm}

Let $U(d,0):=\md\setminus D_{\z_d}$.  Then $\Phi$ is well-defined over $U(d,0)$.  Let $V(d,0,d):=\Phi(U(d,0))$.  Then we have $U(d,0)\xlongrightarrow[\cong]{\Phi} V(d,0,d)$.  Fourier transform also behaves well in flat families as in the following lemma.  

We say a sheaf $\mf\in M(d,0)$ ($W(r,0,n)$, $U(d,0)$, $V(d,0,d)$, resp.) if the S-equivalence class of $\mf$ is in $M(d,0)$ ($W(r,0,n)$, $U(d,0)$, $V(d,0,d)$, resp.).
\begin{lemma}\label{ftinf}Let $S$ be a Noetherian scheme.  $\tau:S\times \p^2\ra\p^2$ is the projection. 

(1) Let $\mf$ be a sheaf over $S\times \p^2$ which is a $S$-flat family of pure sheaves of class $u_d$ such that $H^1(\mf_s\otimes\mo_{\p^2}(2))=0$ for all $s\in S$.  Then the Fourier transform $$\mg_{\mf}:=(id_S\times q)_*((id_S\times p)^*(\mf\otimes \tau^{*}\mo_{\p^2}(2)))\otimes \tau^{*}\mo_{\p^2}(-1)$$ is a flat $S$-family of sheaves of class $c^d_d$.  In particular if $\mf_s\in U(d,0)$ then $(\mg_{\mf})_s\in V(d,0,d)$.

(2) Let $\mg$ be a flat $S$-family of torsion-free sheaves of class $c^d_d$ such that $\mg_s|_l\cong\mo_l^{\oplus d}$ for all $s\in S$ where $l\in |H|$ is a generic line.  Then the Fourier transform
$$\mf_{\mg}:=R^1(id_S\times q)_*((id_S\times p)^*(\mg\otimes \tau^{*}\mo_{\p^2}(-1)))\otimes \tau^{*}\mo_{\p^2}(-1)$$
is a $S$-flat family of sheaves of class $u_d$.  In particular, if $\mg_s\in V(d,0,d)$ then $(\mf_{\mg})_s\in U(d,0)$.
\end{lemma} 
\begin{proof}We only show Statement (1) since (2) is analogous.  It is enough to show the $S$-flatness of $\mg_{\mf}$.  By Proposition 2.1.2 in \cite{hl}, it suffices to show that $h_*(\mg_{\mf}\otimes \tau^*\mo_{\p^2}(n))$ is locally free over $S$ for $n\gg0$, where $h:S\times \p^2\ra S$ is the projection.  We have the following commutative diagram
\[\xymatrix{S&S\times \cd\ar[l]_{h^{\cd}}\ar @<0.5ex> [r]^{id_S\times p}  \ar @<-0.5ex> [r]_{id_S\times q}\ar[d]_{\tau^{\cd}}\ar[d] &S\times\p^2\ar[d]^{\tau}\ar[rd]^{h}&\\ & \cd \ar @<0.5ex> [r]^p\ar @<-0.5ex> [r]_q &\p^2&S.}\]
$h^{\cd}=h\circ (id_S\times p)=h\circ (id_S\times q)$, $p\circ \tau^{\cd}=\tau\circ(id_S\times p)$ and $q\circ\tau^{\cd}=\tau\circ(id_S\times q)$.  
Hence 
\[\mg_{\mf}\cong (id_S\times q)_*((id_S\times p)^*(\mf\otimes \tau^*\mo_{\p^2}(2))\otimes (\tau^{\cd})^* q^*\mo_{\p^2}(-1))\]
and
\[\begin{array}{r}h_*(\mg_{\mf}\otimes \tau^*\mo_{\p^2}(n))\cong h^{\cd}_{*}((id_S\times p)^*(\mf\otimes \tau^*\mo_{\p^2}(2))\otimes (\tau^{\cd})^* q^*\mo_{\p^2}(n-1))\\
\cong h_{*}(id_S\times p)_*((id_S\times p)^*(\mf\otimes \tau^*\mo_{\p^2}(2))\otimes (\tau^{\cd})^* q^*\mo_{\p^2}(n-1)).\end{array}\]
Because $\tau$ is flat, $ (id_S\times p)_*(\tau^{\cd})^* q^*\mo_{\p^2}(n-1)\cong \tau^*p_* q^*\mo_{\p^2}(n-1)$ is locally free on $S\times \p^2$ and $R^i(id_S\times p)_*(\tau^{\cd})^* q^*\mo_{\p^2}(n-1)=0$ for all $i\neq0$ and $n\geq0$.  Hence by Lemma \ref{prof} we have 
\[\begin{array}{l}(id_S\times p)_*((id_S\times p)^*(\mf\otimes \tau^*\mo_{\p^2}(2))\otimes (\tau^{\cd})^* q^*\mo_{\p^2}(n-1))\\ \cong \mf\otimes \tau^*\mo_{\p^2}(2)\otimes (id_S\times p)_*(\tau^{\cd})^* q^*\mo_{\p^2}(n-1),~for~all~n\geq0 \\
\cong \mf\otimes \tau^*\mo_{\p^2}(2)\otimes \tau^*p_* q^*\mo_{\p^2}(n-1),~for~all~n\geq0.\end{array}\]
Hence for $n\geq0$,
\[\begin{array}{ll}h_*(\mg_{\mf}\otimes \tau^*\mo_{\p^2}(n))&\cong h_{*}(\mf\otimes \tau^*\mo_{\p^2}(2)\otimes \tau^*p_* q^*\mo_{\p^2}(n-1))\\
&\cong h_{*}(\mf\otimes \tau^*(\mo_{\p^2}(2)\otimes p_* q^*\mo_{\p^2}(n-1))).\end{array}\]
$p_*q^*\mo_{\p^2}(n-1)$ is locally free over $\p^2$ and we have the following sequence
\begin{equation}\label{option}0\ra\mo_{\p^2}(-1)^{\oplus h^0(\mo_{\p^2}(n-2))}\ra\mo_{\p^2}^{\oplus h^0(\mo_{\p^2}(n-1))}\ra p_*q^*\mo_{\p^2}(n-1)\ra 0.\end{equation}
$\mf\otimes \tau^*(\mo_{\p^2}(2)\otimes p_* q^*\mo_{\p^2}(n-1))$ is $S$-flat because so is $\mf$.  Since for all $s\in S$, we have $H^1(\mf_s\otimes\mo_{\p^2}(2))=0$.  Therefore by (\ref{option}) we have $H^1(\mf_s\otimes \mo_{\p^2}(2)\otimes p_* q^*\mo_{\p^2}(n-1))=0$ for all $s\in S$ and hence $h_{*}(\mf\otimes \tau^*(\mo_{\p^2}(2)\otimes p_* q^*\mo_{\p^2}(n-1)))$ is locally free on $S$ for all $n\geq0$.  

The lemma is proved.
\end{proof}

\begin{lemma}\label{prof}Let $f:X\ra Y$ be a flat morphism of ringed spaces.  Let $F$ be a coherent sheaf over $Y$ with finite homological dimension.  Let $\ma$ be a sheaf of $\mo_X$-modules over $X$ such that $R^{i}f_*\ma=0$ except for $i=i_0$ and $\Tor^j_{\mo_X}(\ma,f^*F)=\Tor^j_{\mo_Y}(R^{i_0}f_*\ma, F)=0$ for all $j>0$.  Then we have
\begin{equation}\label{lonely}\forall~i,~R^if_*(f^*F\otimes \ma)\cong F\otimes R^{i}f_*\ma.\end{equation}
\end{lemma}
\begin{proof}Since $F$ is of finite homological dimension, we can take a locally free resolution of it of finite length
\[E^{\bullet}\ra F.\]
then the sequence
\[f^*E^{\bullet}\otimes \ma\ra f^*F\otimes \ma\]
is still exact since $f$ is flat, $E^n$ are locally free and $\Tor^j_{\mo_X}(\ma,f^*F)=0$ for all $j>0$.

On the other hand, for any $n,i$,
$R^if_*(f^*E^n\otimes \ma)\cong E^n\otimes R^if_*\ma$ by the projection formula.  Thus for $i\neq i_0$, $R^if_*(f^*E^n\otimes \ma)=0$ for all $n$.  Hence $R^{\bullet}f_*(f^*F \otimes \ma)$ can be computed by the sequence $R^{i_0}f_*(f^*E^{\bullet}\otimes \ma)\cong E^{\bullet}\otimes R^{i_0}f_*\ma$.  

Hence (\ref{lonely}) holds because the sequence 
\[E^{\bullet}\otimes R^{i_0}f_*\ma\ra F\otimes R^{i_0}f_{*}\ma\]
is still exact by $\Tor^j_{\mo_Y}(R^{i_0}f_*\ma,F)=0$ for all $j>0$.  Hence the lemma. 
\end{proof}
\begin{lemma}\label{tor}Let $X$ a projective surface, $S$ be any scheme of finite type.  Let $\mg_1,\mg_2$ be two $S$-flat families of torsion-free sheaves on $X$, and let $\mf_1,\mf_2$ be two $S$-flat families of 1-dimensional pure sheaves.  Moreover we ask that for a generic $s\in S$, the intersection of supports of $(\mf_1)_s$ and $(\mf_2)_s$ is of dimension 0.  Then 
\begin{enumerate}
\item[(1)] $\Tor^j_{\mo_{X\times S}}(\mg_1,\mg_2)=0$ for all $j>0$;
\item[(2)] $\Tor^j_{\mo_{X\times S}}(\mg_k,\mf_l)=0$ for all $j>0$, $k,l=1,2$;
\item[(3)] $\Tor^j_{\mo_{X\times S}}(\mf_1,\mf_2)=0$ for all $j>0$.
\end{enumerate}
\end{lemma}
\begin{proof}$\mg_i$ and $\mf_i$ are all of homological dimension 1 for $i=1,2$.  We Hence we only need to show $\Tor^1_{\mo_{X\times S}}(\mg_1,\mg_2)=\Tor^1_{\mo_{X\times S}}(\mf_k,\mg_l)=\Tor^1_{\mo_{X\times S}}(\mf_1,\mf_2)=0$.

$\Tor^1_{\mo_{X\times S}}(\mg_1,\mg_2)$ is a subsheaf of $\mg_1\otimes\me$ for some $\me$ locally free.  But on the other hand the support of $\Tor^1_{\mo_{X\times S}}(\mg_1,\mg_2)$ consists of all points $(x,s)$ such that not both $(\mg_1)_s$ and $(\mg_2)_s$ are locally free at $x$.  Hence $\Tor^1_{\mo_{X\times S}}(\mg_1,\mg_2)$ is a torsion sheaf while it is also a subsheaf of the torsion-free sheaf $\mg_1\otimes\me$, and hence $\Tor^1_{\mo_{X\times S}}(\mg_1,\mg_2)=0$. 

Analogous arguments apply to $\Tor^1_{\mo_{X\times S}}(\mf_k,\mg_l)$ and $\Tor^1_{\mo_{X\times S}}(\mf_1,\mf_2)$ and hence the lemma.
\end{proof}

Let $\mf^S$ ($\mf^T$, resp.) be a $S$-flat ($T$-flat, resp.) family of sheaves in $U(d,0)$ ($U(r,0)$, resp.), and let $\mg^S$ ($\mg^T$, resp.) be the Fourier transform of $\mf^S$ ($\mf^T$, resp.).  Then by Lemma \ref{ftinf}, $\mg^S$ ($\mg^T$, resp.) is a $S$-flat ($T$-flat, resp.) family of sheaves in $V(d,0,d)$ ($V(r,0,r)$, resp.)

Denote by $\z^r_d(r)|_S$ ($\z^d_r(d)|_T$, resp.) the pullback of $\z^r_d(r)$ ($\z^d_r(d)$, resp.) on $U(d,0)$ ($U(r,0)$, resp.) via the classifying map $S\ra U(d,0)$ ($T\ra U(r,0)$) induced by $\mf^S$ ($\mf^T$, resp.), and denote by $\lambda_d(r)|_S$ ($\lambda_r(d)|_T$, resp.) the pullback of $\lambda_d(r)$ ($\lambda_r(d)$, resp.) on $V(d,0,d)$ ($V(r,0,r)$, resp.) via the classifying map $S\ra V(d,0,d)$ ($T\ra V(r,0,r)$) induced by $\mg^S$ ($\mg^T$, resp.).  Define
$$\rd^1:=\big\{(s,t)\in S\times T\big| H^0(\mf^S_s\otimes \mg^T_t)\neq 0\big\};$$
$$\rd^2:=\{(s,t)\in S\times T\big|H^0(\mg^S_s\otimes\mf^T_t)\neq 0\}.$$
Then according to Theorem 2.1 in \cite{nila}, $\rd^1$ ($\rd^2$, resp.) is a divisor of line bundle $\z^r_d(r)|_S\boxtimes \lambda_r(d)|_T$ ($\lambda_d(r)|_S\boxtimes \z^d_r(d)|_T$, resp.), which induces a map $\zeta_{\rd^1}$ ($\zeta_{\rd^2}$, resp.) as follows.
\[H^0(S,\z^r_d(r)|_S)^{\vee}\xlongrightarrow{\zeta_{\rd^1}} H^0(T,\lambda_r(d)|_T);\]
\[H^0(S,\lambda_d(r)|_S)^{\vee}\xlongrightarrow{\zeta_{\rd^2}} H^0(T,\z^d_r(d)|_T).\]
\begin{prop}\label{ft}$\z^r_d(r)|_S\boxtimes \lambda_r(d)|_T\cong\lambda_d(r)|_S\boxtimes \z^d_r(d)|_T$ and $\rd^1=\rd^2$.  In particular, we have the following commutative diagram
\begin{equation}\xymatrix{H^0(S,\z^r_d(r)|_S)^{\vee}\ar[r]^{\zeta_{\rd^1}}\ar[d]_{\cong}&H^0(T,\lambda_r(d)|_T)\ar[d]^{\cong}\\
H^0(S,\lambda_d(r)|_S)^{\vee}\ar[r]_{\zeta_{\rd^2}} &H^0(T,\z^d_r(d)|_T). }
\end{equation}
\end{prop}
\begin{proof}We have the following commutative diagram
\[\xymatrix@C=1.2cm{S\times \cd\ar @<0.5ex> [d]^{id_S\times q}  \ar @<-0.5ex> [d]_{id_S\times p}& S\times T\times\cd \ar[l]_{\alpha_{S}^{\cd}}\ar @<0.5ex> [d]^{id_{S\times T}\times q}  \ar @<-0.5ex> [d]_{id_{S\times T}\times p}\ar[r]^{\alpha_T^{\cd}} & T\times\cd \ar @<0.5ex> [d]^{id_T\times q}  \ar @<-0.5ex> [d]_{id_T\times p}\\ S\times\p^2\ar[d]_{\tau_S}& S\times T\times \p^2\ar[l]^{\alpha_S}\ar[ld]^{\tau_{S\times T}}\ar[r]_{\alpha_T}\ar[d]_{h}\ar[rd]_{\tau_{S\times T}} & T\times\p^2 \ar[d]^{\tau_T}\\ \p^2 & S\times T&\p^2.}\]
By the basic property of the determinant line bundles (see Ch 8 Theorem 8.1.5 in \cite{hl}) we have 
\[\z^r_d(r)|_S\boxtimes \lambda_r(d)|_T\cong det^{-1}(R^{\bullet}h_{*}(\alpha_S^*\mf^S\otimes\alpha_T^*\mg^T)),\]
and 
\[\lambda_d(r)|_S\boxtimes \z^d_r(d)|_T\cong det^{-1}R^{\bullet}h_{*}(\alpha_S^*\mg^S\otimes\alpha_T^*\mf^T).\]

Let 
$$\widetilde{\mg}:=((id_{S\times T}\times p)^*\alpha_S^*(\mg^S\otimes \tau_S^*\mo_{\p^2}(-1)))\otimes ((id_{S\times T}\times q)^*\alpha_T^*(\mg^T\otimes \tau_T^*\mo_{\p^2}(-1))).$$
Then by Lemma \ref{tor} and Lemma \ref{prof} we have 
\[R^i(id_{S\times T}\times p)_{*}\widetilde{\mg}=0,~\forall~i\neq1,~R^1(id_{S\times T}\times p)_{*}\widetilde{\mg}\cong \alpha_S^{*}\mg^S\otimes\alpha_T^{*}\mf^T,\]
and 
\[R^i(id_{S\times T}\times q)_{*}\widetilde{\mg}=0,~\forall~i\neq1,~R^1(id_{S\times T}\times q)_{*}\widetilde{\mg}\cong \alpha_S^{*}\mf^S\otimes\alpha_T^{*}\mg^T.\]

On the other hand, $h^{\cd}:=h\circ (id_{S\times T}\times p)=h\circ (id_{S\times T}\times q)$.  Hence

$$\z^r_d(r)|_S\boxtimes \lambda_r(d)|_T\cong det(R^{\bullet}h^{\cd}_{*}\widetilde{\mg})\cong \lambda_d(r)|_S\boxtimes \z^d_r(d)|_T.$$

Notice that $\forall~(s,t)\in S\times T$, $\widetilde{\mg}_{s,t}=p^*\mg_s^S(-1)\otimes q^*\mg_t^T(-1)$ is of Euler characteristic zero.  This is because $R^ip_{*}\widetilde{\mg}_{s,t}=0,~\forall~i\neq 1$ and $R^1p_{*}\widetilde{\mg}_{s,t}\cong \mf^S_s\otimes\mg_t^T$ is of Euler characteristic zero.  Also we know that $H^i(\widetilde{\mg}_{s,t})\cong H^{i-1}( \mf^S_s\otimes\mg_t^T)\cong H^{i-1}( \mg^S_s\otimes\mf_t^T)$.  Let
\[\rd^{\cd}:=\big\{(s,t)\in S\times T\big| H^1(\cd,p^*\mg_s^S(-1)\otimes q^* \mg_t^T(-1))\neq 0\big\}.\]
Then $\rd^{\cd}=\rd^1=\rd^2$ set-theoretically.
To complete the proof of the proposition, We only need to construct a section of $det(R^{\bullet}h^{\cd}_{*}\widetilde{\mg})$ defining $\rd^{\cd}$ as a divisor.
Take an exact sequence on $S\times \p^2$ 
\[0\ra\ma^S\ra \mb^S\ra\mg^S\otimes\tau_S^*\mo_{\p^2}(-1)\ra 0,\]
such that $\mb^S\cong \tau_S^*\mo_{\p^2}(-m_S)^{\otimes V_S}$ with $m_S\gg 0$ and $V_S$ some vector space, and $\ma^S$ is locally free.  Then on $S\times T\times \cd$ we have
\[0\ra\widetilde{\ma}\ra\widetilde{\mb}\ra\widetilde{\mg}\ra 0,\]
where $\widetilde{\ma}\cong ((id_{S\times T}\times p)^*\alpha_S^*(\ma^S))\otimes ((id_{S\times T}\times q)^*\alpha_T^*(\mg^T\otimes \tau_T^*\mo_{\p^2}(-1)))$, and
$\widetilde{\mb}\cong ((id_{S\times T}\times p)^*\alpha_S^*(\mb^S))\otimes ((id_{S\times T}\times q)^*\alpha_T^*(\mg^T\otimes \tau_T^*\mo_{\p^2}(-1)))$.  We claim that
$R^ih^{\cd}_*\widetilde{\ma}=R^ih^{\cd}_*\widetilde{\mb}=0,~\forall~i\neq 2$ and both $R^2h^{\cd}_*\widetilde{\ma}$ and $R^2h^{\cd}_*\widetilde{\mb}$ are locally free of the same rank.  If this is true, then $R^{\bullet}h_{*}^{\cd}\widetilde{\mg}$ can be computed by the following sequence
\[R^2h^{\cd}_*\widetilde{\ma}\xlongrightarrow{\eta}R^2h^{\cd}_*\widetilde{\mb},\]
i.e. $R^i h_{*}^{\cd}\widetilde{\mg}=0,~\forall~i\neq 1,2$, $R^1 h_{*}^{\cd}\widetilde{\mg}\cong Ker(\eta)$ and $R^2 h_{*}^{\cd}\widetilde{\mg}\cong Coker(\eta)$.  Therefore $det(\eta)$ gives a section of $det(R^{\bullet}h^{\cd}_{*}\widetilde{\mg})$ defining $\rd^{\cd}$.

To show $R^ih^{\cd}_*\widetilde{\ma}=R^ih^{\cd}_*\widetilde{\mb}=0,~\forall~i\neq 2$, it is enough to show $H^i(p^{*}\ma^S_s\otimes q^*(\mg_t^T(-1)))=H^i(p^{*}\mb^S_s\otimes q^*(\mg_t^T(-1)))=0$ for all $(s,t)\in S\times T$ and $i\neq 2$.  Since $R^j p_*(q^*\mg_t^T(-1))=0$ for all $j\neq1$ and $R^1p_*(q^*\mg_t^T(-1))\cong \mf^T_t(1)$, we have
\[H^i(\widetilde{\ma}_{s,t})=H^i(p^{*}\ma^S_s\otimes q^*(\mg_t^T(-1)))=H^{i-1}(\ma^S_s\otimes \mf^T_t(1));\]
\[H^i(\widetilde{\mb}_{s,t})=H^i(p^{*}\mb^S_s\otimes q^*(\mg_t^T(-1)))=H^{i-1}(\mb^S_s\otimes \mf^T_t(1)).\] 
Hence $H^i(p^{*}\ma^S_s\otimes q^*(\mg_t^T(-1)))=H^i(p^{*}\mb^S_s\otimes q^*(\mg_t^T(-1)))=0,~ \forall~i\neq 2$ by $\mb^S_s\cong\mo_{\p^2}(-m_S)^{\otimes V_S} $ with $m_S\gg0$.

$R^2h^{\cd}_*\widetilde{\ma}$ and $R^2h^{\cd}_*\widetilde{\mb}$ are of the same rank because $\widetilde{\mg}_{s,t}$ is of Euler characteristic zero and hence $\widetilde{\ma}_{s,t}$ and $\widetilde{\mb}_{s,t}$ are of  the same Euler characteristic.  The proposition is proved.
\end{proof}
For any $d,r,n>0$, we have the strange duality map
\begin{equation}\label{dsdm}SD_{d,c^r_n}:H^0(M(d,0),\z^r_d(n))^{\vee}\ra H^0(\wrn,\lcd).\end{equation}  Denote it simply by $SD_{d,(r)}$ if $r=n$.  Notice that the dual map $SD^{\vee}_{d,(2)}$ is the map $SD$ in (\ref{sdm}).
Proposition \ref{ft} implies the following corollary.
\begin{coro}\label{com}The Fourier transform $\Phi$ in (\ref{bcor}) identifies the determinant line bundle $\z_d^r(r)$ on $U(d,0)$ with $\lambda_d(r)$ on $V(d,0,d)$ for any $d,r>0$.  Moreover we have a commutative (up to scalars) diagram as follows
\begin{equation}\xymatrix@C=1.2cm{H^0(U(d,0),\z^r_d(r))^{\vee}\ar[r]^{SD_{d,(r)}}\ar[d]_{\cong}&H^0(V(r,0,r),\lambda_r(d))\ar[d]^{\cong}\\
H^0(V(d,0,d),\lambda_d(r))^{\vee}\ar[r]_{SD_{r,(d)}^{\vee}} &H^0(U(r,0),\z^d_r(d)). }
\end{equation}
\end{coro}

\begin{prop}\label{din}Let $d=1,2$.  Then the strange duality map $SD_{d,(c^r_n)}$ in (\ref{dsdm}) is injective for any $r>0$ and $n=ar$ with $a\in\mathbb{Z}_{\geq 1}$.
\end{prop}
\begin{proof}Let $G$ be a sheaf in $\wrn$.  By Theorem 2.1 in \cite{nila} we can associate to $G$ a section $s_G$ of $\z^r_d(n)$ given by the divisor $D_G:=\{F\in\md|H^0(F\otimes G)\neq0\}$.  To show that $SD_{d,(r,0,n)}$ is injective, it is enough to show that we can find a collection of finitely many sheaves $\{G_i\}$ in $W(r,0,ar)$ such that $\{s_{G_i}\}$ spans $H^0(\md,\z_d^r(ar))$.   

For $d=1,2$, $M(d,0)\cong \ls$ and $\z_d^r(n)\cong\mo_{\ls}(n)$ for all $r>0$ (see \cite{nila} or Proposition 4.1.1 in \cite{yuan}).  We choose a finite collection of distinct points $\{x_j\}_{j\in J}$, and associate to each point $x_j$ a divisor consisting of curves passing through $x_j$, which gives a section $t_j$ of $\mo_{\ls}(1)$.  Let $N=dim~\ls$, then it is possible to choose $N+1$ distinct points $x_j$ such that $\{t_j\}_{j=1}^{N+1}$ spans $H^0(\mo_{\ls}(1))$.  Hence we can choose $n(N+1)$ distinct points $x^k_j$ with $1\leq j\leq N+1$, $1\leq k\leq n$ such that $\{t_{j_1,\cdots,j_n}\}$ spans $H^0(\mo_{\ls}(n))$, where $t_{j_1\cdots,j_n}$ is defined as follows.

\[t_{j_1,\cdots,j_n}:=\prod_{k=1}^n t^k_{j_k},~with~t^k_{j_k} the~section~associated~to~x^k_{j_k}.\]

Let $n=ar$, then we define a collection of semistable sheaves $\{G_{i}\}$ consisting of all the sheaves of the form $\displaystyle{\bigoplus_{l=1}^r} ~I_{Z_{l}}$, where $I_{Z_l}$ is the ideal sheaf of $Z_l$ and $Z_l$ consists of $a$ distinct points in $\{x^k_j\}$.  Let $F\in \md$ with $d=1,2$.  Then $H^0(F\otimes G)\neq0$ for $G=\displaystyle{\bigoplus_{l=1}^r} ~I_{Z_l}$ iff $Supp(F)\cap(\displaystyle{\bigcup_{l=1}^r}~ Z_l))\neq\emptyset$.  Hence $\{s_{G_i}\}$ spans $H^0(\mo_{\ls}(ar))=H^0(\md,\z_d^r(ar))$.  Hence the proposition.
 \end{proof}
 \begin{rem}\label{lunch}We have the restriction map $$H^0(W(r,0,r),\lambda_r(d))\xrightarrow{\jmath_r} H^0(V(r,0,r),\lambda_r(d)).$$
Then $\jmath_r\circ SD_{d,(r)}$ is injective for $d=1,2$ since $G_i=\displaystyle{\bigoplus_{l=1}^r}~ I_{Z_{l}}\in V(r,0,r)$. 
 \end{rem}
\begin{proof}[Proof of Proposition \ref{inje}]According to Corollary \ref{com} and Theorem \ref{lp}, we have the following commutative diagram
\begin{equation}\xymatrix@C=1.2cm{H^0(M(2,0),\z^d_2(d))^{\vee}\ar[r]^{SD_{2,(d)}}\ar[d]_{\cong}& H^0(W(d,0,d),\lambda_d(2))\ar[r]^{\jmath_d} &H^0(V(d,0,d),\lambda_d(2))\ar[d]^{\cong}\\
H^0(W(2,0,2),\lambda_2(d))^{\vee}\ar[r]_{SD=SD_{d,(2)}^{\vee}} &H^0(\md,\z^2_d(2))\ar[r]_{\imath_d}&H^0(U(d,0),\z^2_d(2)). }
\end{equation}
Then the injectivity of $SD$ follows from Remark \ref{lunch}.
\end{proof}
\begin{rem}\label{expan3}For $d\geq 3$, actually the map $\Phi$ can be extended to a larger open subset than $U(d,0)$ and we have the following isomorphism for all $d,r$
$$\Phi^{*}:H^0(W(d,0,d),\lambda_d(r))\xrightarrow{\cong}H^0(\md,\z^r_d(r)).$$(See Theorem \ref{App1} in the appendix.)
\end{rem}

\section{Surjectivity of the strange duality map.}
We already know that the map $SD$ is injective, to prove the surjectivity, it is enough to show that 
\begin{equation}\label{less}h^0(\md,\z_d^2(2))\leq h^0(W(2,0,2),\lambda_2(d)).\end{equation}
We already know that 
\[h^0(W(2,0,2),\lambda_2(d))=h^0(M(2,0),\z_2^d(d))={5+d \choose d}.\] 
\begin{lemma}\label{smalld}For $d\leq 3$, $h^0(\md,\z_d^2(2))=\left(\begin{array}{c}5+d\\d\end{array}\right)$.
\end{lemma}
\begin{proof}We can find this result in \cite{nila} or use Proposition 4.1.1 and Theorem 4.4.1 in \cite{yuan} to compute directly.
\end{proof}

Let $\mdi$ be the open subscheme of $\md$ parameterzing sheaves with integral supports.  The following lemma follows from Theorem 4.17, Example 4.18 (1) and Remark 6.3 in \cite{yyfifth}. 
\begin{lemma}\label{coding}The codimension of $\md\setminus\mdi$ is $\geq \min\{d-1,7\}$, for all $d$. 
\end{lemma}
Now we assume $d\geq 4$.  Then S-equivalent classes of sheaves with non-integral supports form a subset of codimension $\geq$ 3 in $\md$.  

Recall that there is a unique divisor $D_{\z_d}$ associated to the line bundle $\z_d$.  We have the  following exact sequence.
\begin{equation}\label{zeta}0\ra \z^{r-1}_d(n)\ra \z^r_d(n)\ra \z^r_d(n)|_{D_{\z_d}}\ra0,~for~all~n,r.\end{equation}

Recall that we have a projection $\pi:\md\ra\ls$ sending every sheaf to its support.  By Theorem 4.3.1 in \cite{yuan}, we have $\pi_{*}\z_d\cong \mo_{\ls}$.
\begin{prop}\label{house}$R^i\pi_{*}\z_d^r=0$ for all $i>0$, $r>0$.
\end{prop}
\begin{proof}By Proposition 3.0.3 and Proposition 4.2.11 in \cite{yuan}, we know that $\z^r_d(n)$ with $r>0$ has no higher cohomology as $n\gg0$.  We may choose $n$ very large such that $R^i\pi_{*}\z_d^r(n)$ has no higher cohomology for all $i$.  Then we get a surjection $H^i(\z^r_d(n))\twoheadrightarrow H^0(R^i\pi_{*}\z_d^r(n))$ which implies that for $i>0$, $H^0(R^i\pi_{*}\z_d^r(n))=0$ as $n\ra +\infty$.  Therefore $R^i\pi_{*}\z_d^r=0$ for $i>0$.
\end{proof}

Proposition \ref{house} together with the fact that $\pi_{*}\z_d\cong \mo_{\ls}$ imply that $H^i(\z_d(n))=0$ for all $i>0,~n\geq 0$.  So for $n\geq 0$ we have
\begin{equation}\label{gst}h^0(\z_d^2(n))=h^0(\z_d(n))+h^0(\z_d^2(n)|_{D_{\z_d}}).\end{equation}
In general we have
\begin{equation}\label{gsg}h^0(\z^r_d(n))\leq h^0(\z^{r-1}_d(n))+h^0(\z_d^r(n)|_{D_{\z_d}}).\end{equation}

The crucial theorem is as follows.
\begin{thm}\label{main}For $d\geq4$, $r>0$ and $n\geq 0$, we have

(1) $H^0(M(d,0),\z_d^r(n)|_{D_{\z_d}})=0$ if $r>n$;

(2) $h^0(M(d,0),\z_d^r(r)|_{D_{\z_d}})\leq h^0(M(d-3,0),\z^r_{d-3}(r))$.
\end{thm}
Actually, Proposition \ref{house} is not necessary to our proof of the surjectivity.  We will see that given Theroem \ref{main}, (\ref{gsg}) will suffice to get (\ref{less}).  
The following corollary to Theorem \ref{main} is trivial.
\begin{coro}\label{dim}(1) $H^0(M(d,0),\z_d^r(n))\cong H^0(M(d,0),\z_d^n(n))$, for $r>n$.

(2) $h^0(M(d,0),\z_d^r(r))\leq h^0(M(d_0,0),\z_{d_0}^r(r))+\displaystyle{\sum_{ 0\leq k<\frac d3 }} h^0(M(d-3k,0),\z_{d-3k}^{r-1}(r)),$ with $d_0\equiv d ~(mod~3)$ and $1\leq d_0\leq3.$ 
\end{coro}
In order to prove Theorem \ref{main}, we first need to construct some birational maps relating $D_{\z_d}$ and $M(d-3,0)$ to the Hilbert scheme $Hilb^{[\frac{d(d-3)}2]}(\p^2)$ of $\frac{d(d-3)}2$-points on $\p^2$.  The strategy is very similar to \cite{yyfifth}. 

Let $e:=\frac{d(d-3)}2$ and $H_e:=Hilb^{[e]}(\p^2)$.  Denote by $\mi_e$ the universal ideal sheaf over $\p^2\times H_e$.  From now on we take the following convention.

\textbf{Convention.} If we have a product $\p^2\times M$ with $M$ some moduli space (e.g. $M(d,0),W(r,0,n),D_{\z_d},H_e$, etc.), then we usually denote by $q$ the projection $\p^2\times M\ra\p^2$, and $p$ the projection $\p^2\times M\ra M$.  Most of the time, we use $p$ and $q$ without clarifying that they are maps from the product of $\p^2$ with some moduli space.

Let $Q_1:=Quot_{\p^2\times H_e/H_e}(\mi_e\otimes q^{*}\mo_{\p^2}(d-3),dn)$ and $Q_2:=Quot_{\p^2\times H_e/H_e}(\mi_e\otimes q^{*}\mo_{\p^2}(d-3),(d-3)n)$ be the two relative Quot-schemes over $H_e$ parametrizing quotients with Hilbert polynomials $P(n)=dn$ and $P(n)=(d-3)n$ respectively.  Let $\rho_i:Q_i\ra H_e$ be the projection.  
Each point $[f_1:I_Z(d-3)\twoheadrightarrow F_d]\in Q_1$ ($[f_2:I_Z(d-3)\twoheadrightarrow F_{d-3}]\in Q_2$, resp.) over $[I_Z]\in H_e$ must have the kernel $\mo_{\p^2}(-3)$ ($\mo_{\p^2}$, resp.).  This is because $Ker(f_i)$ are torsion free of rank 1 and second Chern class zero.   Therefore by Lemma 5.7 in \cite{yyfifth}, both $F_d$ and $F_{d-3}$ are pure of dimensional one, 
for every $[f_1:I_Z(d-3)\twoheadrightarrow F_d]\in Q_1$ and $[f_2:I_Z(d-3)\twoheadrightarrow F_{d-3}]\in Q_2$.

For any ideal sheaf $I_Z$ with $len(Z)=e$, we have $H^0(I_Z(d-3))\neq0$ and $H^0(I_Z(d))\neq 0$.  Hence $\rho_i$ are surjective.  
We write down the following two exact sequences.
\begin{equation}\label{dd}0\ra \mo_{\p^2}(-3)\ra I_Z(d-3)\ra F_d\ra0;
\end{equation}
\begin{equation}\label{ddt}0\ra\mo_{\p^2}\ra I_Z(d-3)\ra F_{d-3}\ra 0.
\end{equation}
If $F_d$ ($F_{d-3}$, resp.) is semistable, then $F_d\in D_{\z_d}$ ($F_{d-3}\in M(d-3,0)$, resp.).  

We will construct rational maps $g_1:Q_1\dashrightarrow D_{\z_d}$ and $g_2:Q_2\dashrightarrow M(d-3,0)$.  We then will use these two maps to relate $H^0(M(d,0),\z_d^r(n)|_{D_{\z_d}})$ to $H^0(M(d-3,0),\z_{d-3}^r(n))$.  

\begin{flushleft}{\textbf{$\bigstar$\large{ A birational map from $Q_1$ to $D_{\z_d}$.}}}\end{flushleft}  
Choose $m$ large enough.  Let $\Omega_d$ be the smallest open subset of the Quot-scheme $Quot_{\p^2}(\mo_{\p^2}(-m)^{\oplus dm},dn)$ containing all $GL(dm)$-orbits of semistable sheaves and sheaves appearing in $Q_1$.


For any $a,b\in\mathbb{Z}_{\geq 0}$, we define the (locally closed) subscheme $H_e^{a,b}$ of $H_e$ as follows.
$$H_e^{a,b}:=\big\{I_Z\in H_e\big| h^0(I_Z(d-3))=a~and~ h^1(I_Z(d))=b\big\}.$$
$H_e^{a,b}$ is empty unless $a\geq1$.  $\mathscr{H}om_p(q^{*}\mo_{\p^2}(-3),\mi_e\otimes q^*\mo_{\p^2}(d-3)|_{\p^2\times H_e^{a,b}})\cong p_*(\mi_e\otimes q^*\mo_{\p^2}(d)|_{\p^2\times H_e^{a,b}})$ is a locally free sheaf of rank $3d+1+b$ on $H_e^{a,b}$.  Let $Q_1^{a,b}:=\rho_1^{-1}(H_e^{a,b})$, then the following lemma is trivial.
\begin{lemma}\label{hq}$Q_1^{a,b}\cong\p(p_*(\mi_e\otimes q^*\mo_{\p^2}(d)|_{\p^2\times H_e^{a,b}}))$.
\end{lemma}
Define
$$\Omega^{a,b}_d:=\big\{[\mo_{\p^2}(-m)^{\oplus dm}\twoheadrightarrow F_d]\in \Omega_d\big| h^0(F_d)=a~and ~h^1(F_d(3))=b\big\}.$$
We have a universal quotient $\mq_d$ on $\p^2\times\Omega_d$.  Let $\mv^{a,b}:=\mathscr{E}xt_p^1(\mq_d|_{\Omega_d^{a,b}},q^{*}\mo_{\p^2}(-3))$.  Then $\mv^{a,b}$ is a rank $a$ vector bundle over $\Omega^{a,b}_d$.  $\mq_d$ is naturally $GL(dm)$-linearized, hence so is $\mv^{a,b}$.  The projective bundle $\p(\mv^{a,b})$ has a natural $PGL(dm)$-action, and the projection $\p(\mv^{a,b})\ra \Omega^{a,b}$ is $PGL(dm)$-equivariant.  In particular if $a=1$, $\p(\mv^{a,b})=\Omega^{a,b}$.  There is an open subscheme $P^{a,b}_1\subset\p(\mv^{a,b})$ parametrizing torsion free extensions of $\mq_{d,\omega}$ by $\mo_{\p^2}(-3)$ for all $\omega\in \Omega^{a,b}_d$.  Then we have a classifying map $f_1^{a,b}: P_1^{a,b}\ra Q_1^{a,b}$.  (\ref{dd}) implies that $H^0(I_Z(d-3))\cong H^0(\mq_{d,\omega})$ and $H^1(I_Z(d))\cong H^1(\mq_{d,\omega}(3))$ for $[I_Z(d-3)\twoheadrightarrow\mq_{d,\omega}]$.

\begin{lemma}\label{pgl}$f_1^{a,b}: P_1^{a,b}\ra Q_1^{a,b}$ is a principal $PGL(dm)$-bundle.
\end{lemma}
\begin{proof}$GL(dm)$ acts on $\Omega^{a,b}$.  We have $G_{\omega}\cong Aut(\mq_{d,\omega})$ with $G_{\omega}$ the stabilizer of each point $\omega\in \Omega^{a,b}$.  $G_{\omega}$ acts on the fiber space $\mv^{a,b}_{\omega}\cong \Ext^1(\mq_{d,\omega},\mo_{\p^2}(-3))$ of $\mv^{a,b}$ at $\omega$, and $G_{\omega}/\mathbb{C}^{*}$ acts freely on $P_1^{a,b}\cap \p(\mv^{a,b})_{\omega}\cong P_1^{a,b}\cap\p(\Ext^1(\mq_{d,\omega},\mo_{\p^2}(-3)))$.
This gives an interpretation of the $PGL(dm)$-action on $\p(\mv^{a,b})$ and hence $PGL(dm)$ acts freely on $P_1^{a,b}$.  On the other hand the map $f_1^{a,b}$ is $PGL(dm)$-invariant with fiber isomorphic to $PGL(dm)$.  Hence the lemma.
\end{proof}

$\Omega_d$ is a smooth atlas of the moduli stack $\mm_d$ parametrizing semistable sheaves of Hilbert polynomial $dn$ and sheaves appearing in $Q_1$.  
Define
$$\Omega_{d}^{int}:=\big\{[\mo_{\p^2}(-m)^{\oplus dm}\twoheadrightarrow F_d]\in \Omega_d\big| Supp(F_d)~is~integral\big\}.$$
Then by Theorem 4.17 and Example 4.18 (1) in \cite{yyfifth}, the complement of $\Omega_{d}^{int}$ in $\Omega_d$ is of codimension $\geq\min\{d-1,7\}$.  It is easy to see that $\Omega_{d}^{int}$ is smooth and connected, hence $\Omega_d$ is irreducible.

Let $\Omega_d^o=\Omega_d^{1,0}\cap\Omega_{d}^{int}$, then the complement of $\Omega_d^o$ in $\Omega_d^{1,0}$ is of codimension $\geq \min\{d-2,6\}\geq 2$ since $d\geq4$.  
Because $a=1$, $\p(\mv^{1,0})=\Omega^{1,0}$ and $P_1^o:=P_1^{1,0}\times_{\Omega_d^{1,0}}\Omega^{int}_d=\Omega_d^o$.  Let $Q_1^o:=f_1^{1,0}(\Omega_d^o)$.  Then
 the complement of $P_1^o$ ($Q_1^o$, resp.) in $P_1^{1,0}$ ($Q_1^{1,0}$) is of codimension $\geq 2$.  
 
Since sheaves with integral supports are stable, the universal family on $Q_1^o$ induces a morphism $g_1:Q_1^o\ra M(d,0)$ with the image, denoted by $\dzo$, contained in the divisor $D_{\z_d}$.  We have the following commutative diagram
 \begin{equation}\label{tea}\xymatrix{P_1^o\ar[r]^{\cong}\ar[d]_{f_1^{o}}&\Omega_d^o\ar[d]^{f_d^o}\\ Q_1^o\ar[r]_{g_1}& \dzo},\end{equation}
 where $f_1^o$ is the restriction of $f_1^{1,0}$ to $P_1^o$.  Then we have 

\begin{lemma}\label{big}The map $g_1:Q_1^o\ra D_{\z_d}^o$ is an isomorphism.
\end{lemma}
\begin{proof}In (\ref{tea}) both $f_1^o$ and $f_d^o$ are geometric $PGL(dm)$-quotient, hence both $Q_1^o$ and $D_{\z_d}^o$ are the geometric $PGL(dm)$-quotient of $\Omega_d^o$ and $g_1$ is an isomorphism. 
\end{proof}
\begin{lemma}\label{ddc}The complement of $D_{\z_d}^o$ in $D_{\z_d}$ is of codimension $\geq 2$.  In particular since $D^o_{\z_d}$ is smooth, $D_{\z_d}$ is normal.
\end{lemma}
\begin{proof}Let $\dzi:=D_{\z_d}\cap\mdi$.  By Lemma \ref{coding}, the complement of $\dzi$ in $D_{\z_d}$ is of codimension $\geq 2$ for $d\geq4$.  $\dzo\subset\dzi$.
It is enough to show the statement for $D_{\z_d}$ replaced by $D_{\z_d}^{int}$.  All sheaves in $ D_{\z_d}^{int}$ are stable, hence it is enough to show 
\begin{equation}\label{you}dim~\Omega_d^{a,b}\cap\Omega_d^{int}\leq dim~\Omega_d^o-2,~\forall~a\geq1 ~and~(a,b)\neq(1,0).\end{equation}

By Proposition 5.5 in \cite{yyfifth}, we only need to show (\ref{you}) for $(a,b)=(2,0)$.  The complement of $\Omega_d^o$ in $\Omega_d^{1,0}$ is of codimension $\geq 2$, hence by Lemma \ref{hq} and Lemma \ref{pgl} we have 
\[\begin{array}{ll}dim~\Omega_d^o&=dim~\Omega_d^{1,0}\geq dim~P_1^{1,0}=dim~Q_1^{1,0}+dim~PGL(dm)\\ 
&=dim~H_e^{1,0}+3d+dim~PGL(dm)=dim~H_e+3d+dim~PGL(dm).\end{array}\]
The last equation is because $H^{1,0}_e$ is open in $H_e$ by the semi-continuity.
Also by Lemma \ref{hq} and Lemma \ref{pgl}, we have 
\[\begin{array}{ll}dim~\Omega_d^{2,0}\cap\Omega_d^{int}&\leq dim~P_1^{2,0}-1=dim~Q_1^{2,0}+dim ~PGL(dm)-1\\ 
&=dim~H_e^{2,0}+3d+dim~PGL(dm)-1\\&\leq dim~H_e+3d+dim~PGL(dm)-2\leq dim~\Omega_d^o-2.\end{array}\]
Notice that $H_e^{2,0}$ is locally closed in $H_e$ and hence $dim~H_e^{2,0}\leq dim~H_e-1$ because $H_e$ is irreducible.   Hence the lemma.
\end{proof}
\begin{rem}\label{use}We know that there is no universal sheaf over any open subset of $M(d,0)$.  But Lemma \ref{big} implies that there is a universal sheaf over the locally closed subset $D_{\z_d}^o$ in $M(d,0)$.  
\end{rem}
\begin{rem}Our argument for the birationality between $Q_1$ and $D_{\z_d}$ can be simplified, if we use ``stack language" as we did in \cite{yyfifth}.  But we stick to schemes here because we don't want to talk about line bundles over stacks.
\end{rem}
By Lemma \ref{big} we have 
\begin{equation}\label{first}H^0(\z_d^r(n)|_{D^o_{\z_d}})\cong H^0(Q_1^o,g_1^{*}\z_d^r(n)).\end{equation}  We will see that $g_1^{*}\z_d^r(n)$ can be extended to a line bundle $\Lambda^r_n$ on $Q_1^{1,0}$.  

By deformation theory, the relative obstruction space of $Q_1$ over $H_e$ is $\Ext^1(\mo_{\p^2}(-3),F_d)=H^1(F_d(3))$ at the point $[I_Z(d-3)\twoheadrightarrow F_d]$.  Hence $\rho_1$ restricted on $Q_1^{a,0}$ is smooth for any $a\geq 1$.   $H_e^{1,0}$ is open in $H_e$ hence smooth and hence so is $Q_1^{1,0}$.  

There is a universal quotient $\mf_d^{1,0}$ on $\p^2\times Q_1^{1,0}$.  By Lemma \ref{hq} we have an exact sequence as follows.
\begin{equation}\label{uq}0\ra q^{*}\mo_{\p^2}(-3)\otimes p^{*}\mo_{\rho_1}(-1)\ra (id_{\p^2}\times\rho_1)^{*}\mi_e\otimes q^{*}\mo_{\p^2}(d-3)\ra\mf_d^{1,0}\ra 0,
\end{equation}
where $\mo_{\rho_1}(-1)$ is the relative tautological line bundle of the projective bundle $\p(p_*(\mi_e\otimes q^*\mo_{\p^2}(d)|_{\p^2\times H_e^{1,0}}))$.  Let $G_n^r$ be a torsion free sheaf of class $\crn$ on $\p^2$.  Define
\[\Lambda_n^r:=det^{-1}(R^{\bullet}p_{*}(\mf_d^{1,0}\otimes q^{*}G_n^r)).\] 
Then by the universal property of determinant line bundles, $\Lambda_n^r|_{Q_1^o}\cong g_1^{*}\z_d^r(n).$  Since $Q_1^{1,0}$ is smooth and the complement of $Q_1^o$ in $Q_1^{1,0}$ is of codimension $\geq2$, we have
\begin{equation}\label{second}H^0(Q_1^o,g_1^{*}\z_d^r(n))\cong H^0(Q_1^{1,0},\Lambda_n^r).
\end{equation}
On the other hand by (\ref{uq}), we have 
\begin{eqnarray}\label{ee}\Lambda_n^r&\cong& det(R^{\bullet}p_{*}(q^{*}G_n^r(-3)\otimes p^{*}\mo_{\rho_1}(-1)))\otimes det^{-1}( (id_{\p^2}\times\rho_1)^{*}\mi_e\otimes q^{*}G_n^r(d-3))\nonumber\\
&\cong& \mo_{\rho_1}(-1)^{\otimes \chi(G_n^r(-3))}\otimes \rho_1^{*}det^{-1}(R^{\bullet}p_{*}(\mi_e\otimes q^{*}G_n^r(d-3))).
\end{eqnarray}
Notice that the maps $p$ and $q$ at the first line of (\ref{ee}) are from $\p^2\times Q_1^{1,0}$ to $Q_1^{1,0}$ and $\p^2$ respectively, while $p,q$ at the second line are from $\p^2\times H_e^{1,0}$ to $H_e^{1,0}$ and $\p^2$.  As is said in the convention before, we don't change the letters although they are different maps.
 
\begin{proof}[Proof of Statement (1) in Theorem \ref{main}]By (\ref{ee}), $(\rho_{1})_*\Lambda_n^r=0$ if $\chi(G_n^r(-3))>0\Leftrightarrow \chi(G_n^r)=r-n>0$.  By Lemma \ref{ddc} we have an injection $$H^0(D_{\z_d},\z_d^r(n)|_{D_{\z_d}})\hookrightarrow H^0(D^o_{\z_d},\z_d^r(n)|_{D^o_{\z_d}}).$$  Then Statement (1) follows from (\ref{first}) and (\ref{second}). 
\end{proof}

Let $r=n$, then $(\rho_{1})_*\Lambda_r^r\cong det^{-1}(R^{\bullet}p_{*}(\mi_e\otimes q^{*}G_r^r(d-3)))=:L^r_{1,0}$ is a line bundle on $H_e^{1,0}$.  There is an obvious extension of $L^r_{1,0}$ to a line bundle $L^r$ on the whole $H_e$.  
\begin{lemma}\label{hhc}The complement of $H_e^{1,0}$ in $H_e$ is of codimension $\geq2$.
\end{lemma}
We will prove Lemma \ref{hhc} in next subsection.  Because $M(d,0)$ is irreducible, normal and Cohen-Macaulay, Lemma \ref{ddc} and Lemma \ref{hhc} together with (\ref{first}) and (\ref{second}) implies the following proposition.
\begin{prop}\label{boy}$h^0(D_{\z_d},\z_d^r(r)|_{D_{\z_d}})=h^0(H_e,L^r),$ for all $r>0.$
\end{prop}

\begin{flushleft}{\textbf{$\bigstar$\large{ A rational map from $Q_2$ to $M(d-3,0)$.}}}\end{flushleft}  
Again we choose $m$ large enough and let $\Omega_{d-3}$ be the smallest open subset of the Quot-scheme $Quot_{\p^2}(\mo_{\p^2}(-m)^{\oplus (d-3)m},(d-3)n)$ containing all $GL((d-3)m)$-orbits of semistable sheaves and sheaves appearing in $Q_2$.  Let $\mq_{d-3}$ be the universal quotient over $\p^2\times\Omega_{d-3}$.

Define 
\[\Omega^{\leq A}_{d-3}:=\big\{[\mo_{\p^2}(-m)^{\oplus (d-3)m}\twoheadrightarrow F_{d-3}]\in \Omega_{d-3}\big| h^0(F_{d-3})\leq A,~h^0(F_{d-3}(-3))=0\big\};\]
and \[\Omega_{d-3}^o:=\Omega_{d-3}^{int}\cap\Omega^{\leq 1}_{d-3}.\]

\begin{lemma}\label{hi}(1) The complement of $\Omega_{d-3}^{\leq 0}$ is of codimension $\geq 1$.

(2) The complement of $\Omega_{d-3}^o$ in $\Omega_{d-3}$ is of codimension $\geq2$,  if $d-3\neq2$.  
\end{lemma}
\begin{proof}Directly follows from results in \cite{yyfifth} (see Proposition 5.5, Remark 5.6, Theorem 4.17 and Example 4.18 (1) in \cite{yyfifth}).
\end{proof}
$\mathscr{E}xt_p^1(\mq_{d-3},q^{*}\mo_{\p^2})$ is locally free of rank $3(d-3)$ on $\Omega^{\leq A}_{d-3}$.  The projective bundle $\widetilde{P_2^{\leq A}}:=\p(\mathscr{E}xt_p^1(\mq_{d-3}|_{\p^2\times\Omega_{d-3}^{\leq A}},q^{*}\mo_{\p^2}))$ has a natural $PGL((d-3)m)$-action, and the projection $\widetilde{P_2^{\leq A}}\ra \Omega^{\leq A}_{d-3}$ is $PGL((d-3)m)$-equivariant.  Let $P_2^{\leq A}$ be the open subset of $\widetilde{P^{\leq A}_2}$ parametrizing torsion free extensions.  

Define 
$$P_2^o:=P_2^{\leq A}\times_{\Omega_{d-3}^{\leq A}}\Omega_{d-3}^o.$$  
Then $P_2^o$ is a projective bundles over $\Omega_{d-3}^o$.  We have a morphism $f_2^{\leq A}:P_2^{\leq A}\ra Q_2$ induced by the universal extension on $P_2^{\leq A}$.  Denote by $Q_2^{\leq A}$ ($Q_2^o$, resp.) the image of $P^{\leq A}_2$ ($P_2^o$, resp.), and $f^o_2$ the restriction of $f^{\leq A}_2$ to $P_2^o$.  Then we have the following lemma. 
\begin{lemma}\label{pgla}$\forall A\geq0$, $f_2^{\leq A}: P_2^{\leq A}\ra Q_2^{\leq A}$ is a principal $PGL((d-3)m)$-bundle.
\end{lemma}
\begin{proof}The proof is analogous to Lemma \ref{pgl}
\end{proof}
We have a commutative diagram 
\begin{equation}\label{did}\xymatrix@C=1.2cm{P_2^o\ar[r]^{\sigma_2^o}\ar[d]_{f_2^o} & \Omega_{d-3}^o\ar[d]^{f_M^o}\\ Q_2^o\ar[r]_{g_2~~~~~}& M(d-3,0)^o}
\end{equation}
where $f_2^o$ and $f_M^o$ are principal $PGL((d-3)m)$-bundles, $\sigma_2^o$ is $PGL((d-3)m)$-equivariant, and hence it descends to the map $g_2$.  Moreover, $P_2^o$ is a projective bundle over $\Omega_{d-3}^o$, hence $(\sigma_{2}^o)_*\mo_{P_2^o}\cong\mo_{\Omega_{d-3}^o}$ and hence $(g_{2})_*\mo_{Q_2^o}\cong \mo_{M(d-3,0)^o}$.  Hence
\begin{equation}\label{one}H^0(M(d-3,0)^o,\z_{d-3}^r(n))\cong H^0(Q_2^o,g_2^{*}\z_{d-3}^r(n)).\end{equation}  

Define
$$H_e^{\leq A}:=\big\{I_Z\in H_e\big|~h^0(I_Z(d-3))\leq A+1~ and ~h^0(I_Z(d-6))=0\big\}.$$  
Recall the projective map $\rho_2:Q_2\ra H_e$.  By (\ref{ddt}), the preimage of $H_e^{\leq A}$ via $\rho_2$ is exactly $Q_2^{\leq A}$.
By deformation theory and (\ref{ddt}), $\rho_2$ is smooth at the locus 
$$H_e^{sm}:=\big\{ I_Z\in H_e\big| h^0(I_Z(d-3))\leq 1\big\}=\big\{ I_Z\in H_e\big| h^0(I_Z(d-3))=1\big\}.$$  
Notice that $h^0(I_Z(d-3))\geq 1$ for all $I_Z\in H_e$.  Also it is easy to see that $\rho_2$ induces an isomorphism between $H_e^{sm}$ and its preimage.  Hence $Q_2$ and $H_e$ are birational.   $H^{\leq 0}_e\subset H^{sm}_e$.

\begin{lemma}\label{leg} The complement of $H_e^{\leq 0}$ in $H_e$ is of codimension $\leq 2$.  In particular, $Q_2^{\leq A}$ is irreducible containing $Q_2^o$ as a dense open subset for all $A\geq 0$.
\end{lemma}
\begin{proof}By Lemma 5.9 in \cite{yyfifth}, points $I_Z\in H_e$ such that $H^0(I_Z(d-6))\neq 0$ forms a subset of codimension $\geq 3(d-3)\geq 3$.  Hence to show the lemma, it is enough to show that $H_e^{\leq A}\setminus H_e^{\leq 0}$ is of dimension $\leq d(d-3)-2$ for all $A\geq 0$.   The relative dimension of $\rho_2$ is $\geq 1$ over $H_e^{\leq A}\setminus H_e^{\leq 0}$, hence it is enough to show that $Q_2^{\leq A}\setminus Q_2^{\leq 0}$ is of dimension $\leq d(d-3)-1=dim~Q_2^{\leq 0}-1$, which follows from Lemma \ref{hi} (1) and Lemma \ref{pgla} .  

Finally $Q_2^{\leq A}$ is irreducible because $Q_2^{\leq 0}$ is dense open in $Q_2$ and $Q_2^{\leq 0}\cong H_e^{\leq 0}$ hence irreducible.  Hence the lemma.
\end{proof}
\begin{proof}[Proof of Lemma \ref{hhc}]By Lemma 5.10 in \cite{yyfifth}, we know that points $I_Z\in H_e$ such that $H^1(I_Z(d))\neq 0$ forms a subset of codimension $\geq 2$ in $H_e^{\leq 0}$, hence Lemma \ref{hhc} follows directly from Lemma \ref{leg}.
\end{proof}
\begin{lemma}\label{push}$(\rho_{2})_*\mo_{Q_2^{\leq A}}\cong \mo_{H_e^{\leq A}}$ for all $A\geq 0$.
\end{lemma}
\begin{proof}$\rho_2$ is a birational projective morphism.  Both $Q_2^{\leq A}$ and $H_e^{\leq A}$ are integral and $H_e$ is normal, thus the lemma follows from Zariski main theorem.
\end{proof}

On $\p^2\times Q_2^{\leq A}$ we have an exact sequence given by the universal family with $\mf_{d-3}$ the universal quotient.
\begin{equation}\label{uqa}0\ra \mr\ra (id_{\p^2}\times\rho_2)^{*}\mi_e\otimes q^{*}\mo_{\p^2}(d-3)\ra\mf_{d-3}\ra 0.
\end{equation}
The kernel $\mr$ is a line bundle on $\p^2\times Q_2^{\leq A}$.  Let $R:=p_{*}\mr$.  Since $\mr$ restricted to the fiber over each point in $Q_2^{\leq A}$ is isomorphic to $\mo_{\p^2}$, $R$ is a line bundle.  There is a natural map $p^{*}R\ra\mr$, which is injective since $p^{*}R$ is of rank one, and surjective since it is surjective when restricted to the fiber over any point of $Q_2^{\leq A}$.  Hence $\mr\cong p^{*}R=p^{*}(p_{*}\mr)$. 

Let $G_n^r$ be a torsion free sheaf of class $\crn$ on $\p^2$.  Define
\[\Sigma_n^r:=det^{-1}(R^{\bullet}p_{*}(\mf_{d-3}\otimes q^{*}G_n^r)).\] 
Then by the universal property of determinant line bundles, $\Sigma_n^r|_{Q_2^o}\cong g_2^{*}\z_{d-3}^r(n).$  Since $Q_2^{\leq A}$ is irreducible and $Q_2^o$ is open in $Q_2^{\leq A}$, we have
\begin{equation}\label{two}h^0(Q_2^o,g_1^{*}\z_{d-3}^r(n))\geq h^0(Q_1^{\leq A},\Sigma_n^r).
\end{equation}
On the other hand by (\ref{uqa}), we have 
\begin{eqnarray}\label{eea}\Sigma_n^r&\cong& det(R^{\bullet}p_{*}(q^{*}G_n^r\otimes p^{*}R))\otimes det^{-1}( (id_{\p^2}\times\rho_2)^{*}\mi_e\otimes q^{*}G_n^r(d-3))\nonumber\\
&\cong& R^{\otimes \chi(G_n^r)}\otimes \rho_2^{*}det^{-1}(R^{\bullet}p_{*}(\mi_e\otimes q^{*}G_n^r(d-3)))
\end{eqnarray}
Let $r=n$, then by Lemma \ref{push} we have $(\rho_{2})_*\Sigma_r^r\cong det^{-1}(R^{\bullet}p_{*}(\mi_e\otimes q^{*}G_r^r(d-3)))=L^r|_{H_e^{\leq A}}$.  Combine (\ref{one}) (\ref{two}) and Lemma \ref{leg}, we have
\begin{equation}\label{final}h^0(H_e,L^r)\leq h^0(M(d-3,0)^o,\z^r_{d-3}(r)),~\forall r>0.
\end{equation}
\begin{proof}[Proof of Statement (2) in Theorem \ref{main}] For $d-3\neq 2$, by Lemma \ref{coding} and Statement (2) in Lemma \ref{hi}, the complement of $M(d-3,0)^o$ in $M(d-3,0)$ is of codimension $\geq 2$. 
Hence by Proposition \ref{boy} and (\ref{final}) we are done.

Let $d-3=2$.  Notice that in this case $\forall~F_2\in\Omega_2$, $F_2$ is semistable if and only if $h^0(F_2)=0$.  Let $\Omega'_{2}$ parametrizing sheaves $F_2$ such that $h^0(F_2)=0$ and $Supp(F_2)$ is reduced, i.e. not a double line in $|2H|$.  We then have the following diagram analogous to (\ref{did})
\begin{equation}\label{done}\xymatrix@C=1.2cm{P'_2\ar[r]^{\sigma'_2}\ar[d]_{f'_2} & \Omega_2^{'}\ar[d]^{f'_M}\\ Q'_2\ar[r]_{g_2~~~~~}& M(2,0)'}.
\end{equation}
The complement of $M(2,0)'$ is of codimension $\geq 2$ in $M(2,0)$.  Now in (\ref{done}) we still have $f'_2$ a principal $PGL(2m)$-bundle, but $f'_M$ only a good quotient.  $\sigma'_2$ is $PGL(2m)$-equivariant, hence it descends to the map $g_2$.  

What we want is $(g_{2})_*\mo_{Q'_2}\cong \mo_{M(2,0)'}$, and once we have this condition the rest of our argument for $d-3\neq 2$ applies and then we are done.
In order to show $(g_{2})_*\mo_{Q'_2}\cong \mo_{M(d-3,0)'}$, we need to show that $(\sigma'_{2})_*\mo_{P'_2}\cong\mo_{\Omega'_2}$ and this will suffice.  

Let $\widetilde{P_2'}:=\p(\mathscr{E}xt^1_p(\mq_2|_{\p^2\times \Omega'_2},q^*\mo_{\p^2}))$.  Then $P'_2\subsetneq \widetilde{P'_2}$ and $\widetilde{P'_2}$ is a projective bundle over $\Omega'_2$.  We have that $(\widetilde{\sigma'_{2}})_*\mo_{\widetilde{P'_2}}\cong\mo_{\Omega'_2}$.  $\Omega'_2$ is smooth and irreducible and hence so is $\widetilde{P'_2}$.  By a direct observation we see that the complement of $P'_2$ in $\widetilde{P'_2}$ is of codimension $2$ and hence $j_{*}\mo_{P'_2}\cong \mo_{\widetilde{P'_2}}$ with $j:P'_2\hookrightarrow \widetilde{P'_2}$ the open embedding.   On the other hand $\sigma'_2=\widetilde{\sigma'_2}\circ j$, hence $(\sigma'_{2})_*\mo_{P'_2}\cong(\widetilde{\sigma'_{2}})_*(j_{*}\mo_{P'_2})\cong
(\widetilde{\sigma'_{2}})_*\mo_{\widetilde{P'_2}}\cong\mo_{\Omega'_2}$.  

Hence we have proven Theorem \ref{main}.
\end{proof}
\begin{thm}\label{sd}The strange duality map 
\begin{equation}\label{mm}SD:H^0(W(2,0,2),\lambda_2(d))^{\vee}\ra H^0(\md,\z_d^2(2))\end{equation}
is an isomorphism for all $d>0$.
\end{thm}
\begin{proof}By Theorem 4.3.1 in \cite{yuan}, we know that $\pi_{*}\z_d\cong\mo_{\ls}$ for all $d>0$ with $\pi: M(d,0)\ra\ls$. Hence $h^0(M(d,0),\z_d(n))={n+\frac{d(d+3)}2\choose n}$.  By Corollary \ref{dim} and Lemma \ref{smalld} we get the following equation
\[h^0(\md,\z_d^2(2))\leq {5+d \choose d}=h^0(W(2,0,2),\lambda_2(d)).\]
Hence $SD$ is an isomorphism because it is injective by Proposition \ref{inje}.
\end{proof}
\begin{rem}The strange duality conjecture says that the map
\[SD_{c,u}: H^0(M_c,\lambda_u)^{\vee}\ra H^0(M_u,\lambda_c)\]
is an isomorphism.  At numerical level it poses the following equation/question:
\begin{equation}\label{nlsd}h^0(M_c,\lambda_u)= h^0(M_u,\lambda_c)?\end{equation}
which is also called \emph{numerical strange duality}.  However there is another version of this conjecture at the numerical level as follows.
\begin{equation}\label{vnlsd}\chi(\lambda_u)= \chi(\lambda_c)?\end{equation}
with $\chi(-)$ the Euler characteristic.  

Both (\ref{nlsd}) and (\ref{vnlsd}) are true for $X$ a smooth curve.  In the cases studied by Danila (\cite{nilala},\cite{nila}), Abe (\cite{abe}) and also the author and G\"ottsche (\cite{yuan},\cite{GY}), both (\ref{nlsd}) and (\ref{vnlsd}) are right.  However, Theorem \ref{sd} only implies (\ref{nlsd}) for $u=(0,dH,0)$ and $c=(2,0,2)$.  We actually don't know at now whether (\ref{vnlsd}) is true. 
Easy to see that 
$$H^i(W(2,0,2),\lambda_2(d))=H^i(\p^5,\mo_{\p^5}(d))=0,~ for ~all~ i\neq 0~and ~d>0.$$
By Proposition \ref{house}, $H^i(\md,\z^2_d(2))=H^i(|dH|,\pi_*\z^2_d(2))$.  Unfortunately for $d\geq4$ we don't know much on $\pi_*\z^2_d$.  We don't even know whether $\pi_*\z^2_d$ is a vector bundle although it is torsion free of rank $2^{\frac{(d-1)(d-2)}2}$ on $|dH|$.  
\end{rem}
\begin{rem}
The following map might not be surjective for $n>0,d\geq4$ (e.g. see Proposition 5.1 in \cite{Yuan7} for $d=4$)
\[H^0(|dH|,\mo_{|dH|}(n))\times H^0(\z_d^2(2))\ra H^0(\z_d^2(2+n)).\]
By Lemma \ref{big}, Lemma \ref{ddc}, (\ref{first}) and (\ref{second}) we have for $n\geq r$
\[h^0(D_{\z_d},\z_d^r(n)|_{D_{\z_d}})=h^0(Q_1^{1,0},\Lambda_n^r)=h^0(H_e^{1,0},L^r_{1,0}\otimes (\rho_1)_*\mo_{\rho_1}(n-r)).\]
But $(\rho_1)_*\mo_{\rho_1}(n-r)$ is locally free of rank ${3d+1+n-r,\choose n-r}$, hence it is difficult to compute $h^0(H_e^{1,0},L^r_{1,0}\otimes (\rho_1)_*\mo_{\rho_1}(n-r))$ for $n>r$.
So our strategy here may not work on studying $h^0(\z_d^2(n))$ or the following strange duality map for $n>2$
\[SD_{c^2_{n},u_d}:H^0(W(2,0,n),\lambda_{c_{n}^2}(d))^{\vee}\ra H^0(\md,\z_d^2(n)).\]
\end{rem}
\begin{rem}Fourier transform plays an important role in the proof of Theorem \ref{sd}.  This is a limitation which prevents our result from extending to other rational surfaces. 
\end{rem}

\section*{\huge{Appendix.}}
\appendix
\section{More on the Fourier transform.}  
We will use the same notations as in Section 3 and prove the following theorem in this appendix.
\begin{thm}\label{App1}$\Phi:~\md\dashrightarrow W(d,0,d)$ is a birational map of normal projective schemes and $\Phi^{*}\lambda_d(r)\cong \z^r_{d}(r)$, $\forall~d, r$.  

Moreover $\Phi^{*}: H^0(W(d,0,d),\lambda_d(r))\xrightarrow{\cong} H^0(\md,\z^r_{d}(r))$ is an isomorphism.
\end{thm}

With no loss of generality, we assume $d\geq 3$ from now on.  In order to show Theorem \ref{App1}, we need to first extend the map $\Phi$ to a larger subset than $U(d,0)=\md\setminus D_{\z_d}$.

By Lemma 2.2 in \cite{Yuan3}, for every pure sheaf $F$ of class $u_d$, we can assign to it a unique bundle $E_F\cong \displaystyle{\oplus_{i=1}^d}\mo_{\p^2}(m_i)$ with $\displaystyle{\sum_{i=1}^d} m_i=-d$ and have the following exact sequence
\begin{equation}\label{fgen}0\ra E_F\otimes\mo_{\p^2}(-1)\xrightarrow{A} E_F\ra F\ra 0,
\end{equation} 
where $A$ is a $d\times d$ matrix with entries in $\displaystyle{\bigoplus_{n\geq 0}}H^0(\mo_{\p^2}(n))$.  Also by Lemma 3.4 in \cite{Yuan3}, if $F$ is semistable, then $E_F$ must have the form $\displaystyle{\oplus_{i=1}^k}\mo_{\p^2}(n+i)^{\oplus a_i}$ with some integer $n$ and $a_i>0$.  In particular, by direct observation we have 
$$F\in U(d,0)\Rightarrow E_F\cong \mo_{\p^2}(-1)^{\oplus d};$$
$$F\in\md~and~h^0(F)=h^1(F)=1\Rightarrow E_F\cong \mo_{\p^2}\oplus\mo_{\p^2}(-1)^{\oplus d-2}\oplus\mo_{\p^2}(-2).$$

Let $\md^g$ be the open subset consisting of sheaves $F$ such that $H^1(F\otimes\mo_{\p^2}(2))=0$ and $q_{*}(p^{*}(F\otimes\mo_{\p^2}(2)))\otimes\mo_{|H|}(-1)$ are semistable.  By Lemma \ref{ftinf}, $\Phi$ is a well-defined morphism (not only a set-map) over $\md^g$.  By Theorem 4.4 and Theorem 4.8 in \cite{lee}, $\md^g=\md$ for $d\leq 4$.  For $d\geq 5$, we have the following lemma.

\begin{lemma}\label{ftg0}Let $F\in D_{\z_d}$ with $d\geq 5$ such that $h^0(F)=h^1(F)=1$.  If the non-split extension of $F$ by $K_X$ is torsion-free, then $F\in\md^g$.  
In particular, if the support of $F$, denoted by $Supp(F)$, is integral, then $F\in\md^g$. 
\end{lemma}
\begin{proof}Easy to see that $H^1(F\otimes\mo_{\p^2}(2))=0$ because $E_F\cong \mo_{\p^2}\oplus\mo_{\p^2}(-1)^{\oplus d-2}\oplus\mo_{\p^2}(-2)$.  $\Ext^1(F,\mo_{\p^2}(-3))\cong H^1(F)^{\vee}$.  Hence there is a unique non-split extension of $F$ by $\mo_{\p^2}(-3)$ as follows.
\begin{equation}\label{e1}0\ra \mo_{\p^2}(-3)\ra \widetilde I\ra F\ra 0.
\end{equation}
Do Fourier transform to (\ref{e1}) and we get 
\begin{equation}\label{e2}0\ra q_{*}(p^{*}(\widetilde I\otimes\mo_{\p^2}(2)))\otimes\mo_{|H|}(-1)\xrightarrow{\cong}q_{*}(p^{*}(F\otimes\mo_{\p^2}(2)))\otimes\mo_{|H|}(-1)\ra 0,\end{equation}
which is because $q_{*}(p^{*}(\mo_{\p^2}(-1)))=0=R^1q_{*}(p^{*}(\mo_{\p^2}(-1)))$.

$h^0(\widetilde I)\cong h^0(F)=1$, hence there is a non-zero map $\mo_{\p^2}\ra \widetilde I$ which is injective given $\widetilde I$ torsion-free.  Hence we have 
\begin{equation}\label{e3}
0\ra \mo_{\p^2}\ra \widetilde I\ra F_1\ra 0.
\end{equation}
$r(F_1)=0$, $c_1(F_1)=c_1(\widetilde{I})=(d-3)H$, $\chi(F_1)=\chi(F)=0$ and moreover $h^0(F_1)=h^0(\widetilde I)-1=0$ which implies that $F_1$ is semistable, since every subsheaf of $F_1$ can not have positive Euler characteristic.  Hence $F_1\in M((d-3)H,0)\setminus D_{\z_{d-3}}$.

Let $G_1$ be the Fourier transform of $F_1$, then $G_1\in W(d-3,0,d-3)$.  Do Fourier transform to (\ref{e3}) and we get
\begin{equation}\label{e4}
0\ra S^2\mt_{\p^2}(-1) \ra q_{*}(p^{*}(\widetilde I\otimes\mo_{\p^2}(2)))\otimes\mo_{|H|}(-1)\ra G_1\ra 0,
\end{equation}
where $S^2\mt_{\p^2}\cong q_{*}(p^{*}(\mo_{\p^2}(2)))$ is the $2^{\text{nd}}$ symmetric power of the tangent bundle $\mt_{\p^2}$.  $S^2\mt_{\p^2}(-1)\in W(3,0,3)$ and hence by (\ref{e2}) and (\ref{e4}), the Fourier transform of $F$ is semistable and hence $F\in\md^g$.
\end{proof}

Let $W(d,0,d)^g:=\Phi(M(dH,0)^g)$, then $V(d,0,d)=\Phi(U(d,0))\subset W(d,0,d)^g$.  

Define $$\widetilde{U}(d,0):=\big\{F\in\md\big|\begin{array}{c}Supp(F) \text{ is integral, and } \\ h^0(F)=h^1(F)\leq 1.\end{array}\big\}.$$
Then by Lemma \ref{ftg0}, $\widetilde{U}(d,0)\subset \md^g$.
\begin{lemma}\label{ftg1}The complement of $\widetilde{U}(d,0)$ ($V(d,0,d)$, resp.) in $\md$ ($W(d,0,d)$, resp.) is of codimension $\geq 2$.
Hence the complement of $\md^g$ ($W(d,0,d)^g$, resp.) in $\md$ ($W(d,0,d)$, resp.) is of codimension $\geq 2$.
\end{lemma}
\begin{proof}
$\md\setminus\widetilde{U}(d,0)$ is of codimension $\geq2$ is by Theorem 4.17 and Proposition 5.5 in \cite{yyfifth}.  

$V(d,0,d)$ consists of all the sheaves whose restrictions on a generic line $\p^1\cong l\in|H|$ are isomorphic to $\mo_l^{\oplus d}$.  It is enough to show that the following set $\tb$ is of codimension $\geq 2$ in $W(d,0,d)$.
\[\tb:=\big\{G\in W(d,0,d)| H^0(G\otimes\mo_l(-1))\neq 0, ~\forall~l\in |H|.\big\}\]

Let $\widehat{H}$ be the subspace of $H^0(W(d,0,d),\lambda_d(1))$ generated by all the sections induced by sheaves $\mo_l(-1)$ with $l\in |H|$. Also $\widehat{H}$ is the image of $H^0(M(1,0),\z_1^d(d))^{\vee}\cong H^0(|H|,\mo_{|H|}(d))^{\vee}$ via the strange duality map $SD_{1,(d)}$ (def. see (\ref{dsdm})). Notice that $\tb$ is the base locus of $\widehat{H}$.  

$W(d,0,d)$ is of weight zero (def. see \S 1.2 in \cite{Dr2}).  Hence by Theorem B, Theorem E and Theorem F in \cite{Dr2}, we have that $\Pic(W(d,0,d))\cong\mathbb{Z}$ generated by $\lambda_d(1)$. 
Since $\lambda_d(1)$ is effective, it is an ample generator of $\Pic(W(d,0,d))$, and hence every divisor in $|\lambda_d(1)|$ can not be a union of two subdivisors.  Hence either $\tb$ is of codimension $\geq 2$ in $W(d,0,d)$, or $dim~\widehat{H}=1$. 
By Proposition \ref{din}, $SD_{1,(d)}$ is injective and hence $dim~\widehat{H}=h^0(|H|,\mo_{|H|}(d))\geq 3$.  Hence $W(d,0,d)\setminus V(d,0,d)$ is of codimension $\geq 2$.
Hence the lemma. 
\end{proof}

Now we want to show  that $\Phi^{*}\lambda_d(r)\cong \z^r_d(r)$ on $\widetilde{U}(d,0)$ for all $r$.  We need modify Proposition \ref{ft} a bit.  We first prove some lemmas.  

Let $X$ be a projective smooth scheme of dimension $m$ with canonical line bundle $K_X$, for every element $u\in K(X)$, we can find finitely many bundles $E_1,\cdots,E_n$ such that $u=\displaystyle{\sum_{i=1}^n }k_i[E_i]$.  We define $u^{\vee}:=\displaystyle{\sum_{i=1}^n }k_i[E^{\vee}]$ where $E^{\vee}:=\mathscr{H}om(E_i,\mo_X)$, and $u^D:=
u^{\vee}\otimes K_X$. 
By Serre duality, $\chi(c\otimes u)=(-1)^m\chi(c^{\vee}\otimes u^{D})$ for every $c,u\in K(X)$.
\begin{lemma}\label{dual1}Let $X$ be a surface and $S$ be a Noetherian scheme.  $\tau:S\times X\ra X$ and $\nu:S\times S\ra S$ are the projections. 
Let $\mf$ be a sheaf over $S\times X$ which is a $S$-flat family of pure 1-dimensional sheaves of class $u$.  Then 

(1)The sheaf $\mf^{D}:=\mathscr{E}xt^1(\mf,\tau^* K_X)$ is a $S$-flat family of pure 1-dimensional sheaves of class $-u^{D}$.  

(2) For every $s\in S$, $\mf^{D}_s\cong \mathscr{E}xt^1(\mf_s,K_X)$, $Supp(\mf_s)=Supp(\mf_s^{D})$, $\chi(\mf_s)=-\chi(\mf_s^{D})$ and $\mf_s$ is (semi)stable iff so is $\mf^D_s$.

(3) $\forall~c\in K(X)$, the determinant line bundle $\lambda_{\mf}(c):=det^{-1}[R^{\bullet}\nu_*(\tau^*c~\otimes~ [\mf])]$ (see Ch 8.1 in \cite{hl}) is isomorphic to $\lambda_{\mf^{D}}(c^{\vee}):=det^{-1}[R^{\bullet}\nu_*(\tau^* c^{\vee}\otimes~[\mf^{D}])]$.
\end{lemma} 
\begin{proof}$\mf$ is of homological dimension 1 and we can have the following resolution 
\begin{equation}\label{coffe1}0\ra \mk\ra \mh \ra \mf\ra0,\end{equation}
where $\mk$ and $\mh$ are locally free.  Then we have
\begin{equation}\label{coffe2}0\ra \mh^{\vee}\otimes \tau^*K_X\xrightarrow{\alpha} \mk^{\vee}\otimes \tau^*K_X\ra\mf^D\ra0.\end{equation}
$\alpha$ is injective restricted to each fiber of $\nu$ over $s\in S$ hence by Lemma 2.1.4 in \cite{hl} $\mf^D$ is $S$-flat, and $\mf^{D}_s\cong \mathscr{E}xt^1(\mf_s,K_X)$ for all $s\in S$.  The statement on stability is easy to check.  Let $F$ be a pure 1-dimensional sheaf on $X$ and $F^D:=\mathscr{E}xt^1(F,K_X)$.  Then $\forall~F_1$ pure 1-dimensional sheaf,  $F_1\hookrightarrow F$ iff $F^D\twoheadrightarrow F_1^D$.  

For statement (3), we only need to prove it for the case $c=[E]$ with $E$ locally free.  By (\ref{coffe1}) and (\ref{coffe2}), we have
\begin{equation}\label{bread1}\lambda_{\mf}(c)=(det[R^{\bullet}\nu_*(\tau^*E~\otimes~ \mk)])\otimes (det^{-1}[R^{\bullet}\nu_*(\tau^*E~\otimes~ \mh)]);\end{equation} 
\begin{equation}\label{bread2}\lambda_{\mf^{D}}(c^{\vee})=(det^{-1}[R^{\bullet}\nu_*(\tau^*(E^{\vee}\otimes K_X)~\otimes~ \mk^{\vee})])\otimes (det[R^{\bullet}\nu_*(\tau^*(E^{\vee}\otimes K_X)~\otimes~ \mh^{\vee})]).\end{equation} 
On the other hand by Grothendieck duality (or Lemma 5.5 in \cite{abe}), for all $\me$ locally free
$$det^{\otimes (-1)^{dim~X}}[R^{\bullet}\nu_*(\tau^*(K_X)~\otimes~ \me^{\vee})]\cong det^{-1}[R^{\bullet}\nu_*\me].$$ 
Hence the statement.  The lemma is proved.
\end{proof}
\begin{coro}\label{evo1}The map $\kappa:\md\ra\md$ sending $F$ to $F^D$ is an isomorphism and for any determinant line bundle associated to $c\in K(X)$ $\kappa^{*}\lambda_{u_d}(c)\cong\lambda_{u_d}(c^{\vee})$. In particular, $\kappa^{*}\z_d^r(n)\cong\z_d^r(n)$ for all $r,n$.
\end{coro}
\begin{proof}This follows straightforward after Lemma \ref{dual1}.  Notice that $-u_d^D=u_d$ and $(c^r_n)^{\vee}=c^r_n$.
\end{proof}
\begin{rem}\label{de12}If $d=1,2$, then the map $\kappa$ in Corollary \ref{evo1} is an identity.
\end{rem}

Recall that we have the diagram 
\begin{equation}\xymatrix@C=0.01cm{
  \p^2\times |H|~~~~~\supset&\mathcal{D}\ar[d]^q\ar[r]^p
                & \p^2 \\
               &~~~~~~ |H|\cong\p^2&
               },
\end{equation}
where $\mathcal{D}$ is the universal curve in $\p^2\times |H|$.

\begin{lemma}\label{dual2}
(1) Let $F$ be a pure sheaf of class $u_d$ and $Supp(F)$ does not contain any line $l\in |H|$, then its Fourier transform $G_F=q_*(p^*(F\otimes\mo_{\p^2}(2)))\otimes \mo_{|H|}(-1)$ is locally free.  If moreover $h^1(F\otimes\mo_{\p^2}(2))=0$ 
then 
$$G_F^{\vee}\cong q_*(p^*(F^{D}\otimes\mo_{\p^2}(-1)))\otimes \mo_{|H|}(2).$$ 

(2) Let $\mf$ be a sheaf over $S\times \p^2$ which is a $S$-flat family of pure sheaves such that $h^1(\mf_s\otimes\mo_{\p^2}(2))=0$ and $Supp(\mf_s)$ does not contain any line $l\in |H|$ for all $s\in S$.  Then the Fourier transform $$\mg^{\vee}_{\mf}:=(id_S\times q)_*((id_S\times p)^*(\mf^{D}\otimes \tau^{*}\mo_{\p^2}(-1)))\otimes \tau^{*}\mo_{\p^2}(2)$$ is a flat $S$-family of locally free sheaves of class $c^d_d$. 
\end{lemma}
\begin{proof}(1) We have the following resolution on $\mathcal{D}$ by (\ref{fgen}) and the flatness of the map $p$
\begin{equation}\label{fgen2}0\ra p^*(E_F\otimes\mo_{\p^2}(-1))\xrightarrow{p^*A} p^{*}E_F\ra p^{*}F\ra 0.
\end{equation} 
$Supp(F)$ does not contain any line $l\in |H|$, hence $p^*det(A)(l)\neq0$ for every $l\in|H|$ and hence $p^*A$ is injective restricted to fibers of $q$ over all $l\in |H|$.  By Lemma 2.1.4 in \cite{hl},  $p^*(F\otimes\mo_{\p^2}(2))$ is flat over $|H|$ and hence $G_F=q_*(p^*(F\otimes\mo_{\p^2}(2)))\otimes \mo_{|H|}(-1)$ is locally free of rank $d$.  
By $h^1(F\otimes\mo_{\p^2}(2))=0$, we have $h^1((E_F\otimes\mo_{\p^2}(i))|_l)=0$ for $i=1,2$ and any $l\in |H|$.  Hence we have
\begin{equation}\label{qfgen1}0\ra q_*(p^*(E_F\otimes\mo_{\p^2}(1)))\xrightarrow{\widetilde{A}} q_*(p^*(E_F\otimes\mo_{\p^2}(2)))\ra G_F\otimes\mo_{|H|}(1)\ra 0.
\end{equation}
On the other hand we have 
\begin{equation}\label{fgen3}0\ra E_F^{\vee}\otimes\mo_{\p^2}(-3)\xrightarrow{A^T} E_F^{\vee}\otimes \mo_{\p^2}(-2)\ra F^D\ra 0,
\end{equation} 
where $A^T$ is the transform of $A$.  $h^0((E_F^{\vee}\otimes\mo_{\p^2}(-i-2))|_l)=h^1((E_F\otimes\mo_{\p^2}(i))|_l)=0$ for $i=1,2$ and any $l\in |H|$, then we have 
\begin{equation}\label{qfgen2}0\ra q_*(p^*(F^{D}\otimes\mo_{\p^2}(-1)))\ra R^1q_*(p^*(E_F^{\vee}\otimes\mo_{\p^2}(-4)))\xrightarrow{\widetilde{A^T}} q_*(p^*(E_F^{\vee}\otimes\mo_{\p^2}(-3)))\ra 0.
\end{equation}
The relative dualizing sheaf $\omega_{\mathcal{D}/|H|}$ of $q$ is $q^*\mo_{|H|}(1)\otimes p^*\mo_{\p^2}(-2)$.  By Grothendieck duality we have
\[R^1q_*(p^*(E_F^{\vee}\otimes\mo_{\p^2}(-4)))\otimes \mo_{|H|}(1)\cong (q_*(p^*(E_F\otimes\mo_{\p^2}(2))))^{\vee};\]
\[R^1q_*(p^*(E_F^{\vee}\otimes\mo_{\p^2}(-3)))\otimes \mo_{|H|}(1)\cong (q_*(p^*(E_F\otimes\mo_{\p^2}(1))))^{\vee}.\]
It is easy to see that $\widetilde{A^T}=(\widetilde{A})^T$ and (\ref{qfgen2}) can be obtained by tensoring the dual sequence of (\ref{qfgen1}) by $\mo_{|H|}(-1)$.  Hence $G_F^{\vee}\cong q_*(p^*(F^{D}\otimes\mo_{\p^2}(-1)))\otimes \mo_{|H|}(2).$

The proof of (2) is analogous to Lemma \ref{ftinf} (1).  We have finished the proof of the lemma. 
\end{proof}
\begin{lemma}\label{evo2}Let $\varsigma:W(d,0,d)\dashrightarrow W(d,0,d)$ be the birational map sending every locally free $G$ to its dual $G^{\vee}$, then $\varsigma$ induces an isomorphism outside a codimension $\geq2$ subset and $\varsigma^*\lambda_{c_d^d}(u)\cong \lambda_{c_d^d}(-u^D)$ for every $\lambda_{c_d^d}(u)$ well-defined.  In particular, 
$\varsigma^*\lambda_{d}(r)\cong \lambda_{d}(r)$ for all $d,r$.
\end{lemma}
\begin{proof}$V(d,0,d)\cong U(d,0)$.  Sheaves with non-integral supports form a subset of codimension $\geq 2$ in $U(d,0)$ by Theorem 4.17 in \cite{yyfifth}, hence non-locally free sheaves in $V(d,0,d)$ form a subset of codimension $\geq 2$ by Lemma \ref{dual2} (1).  Since $W(d,0,d)\setminus V(d,0,d)$ is of codimension $\geq 2$ by Lemma \ref{ftg1}, non-locally free sheaves in $W(d,0,d)$ form a subset of codimension $\geq 2$.

The statement on determinant line bundles can be shown by analogous argument to Lemma \ref{dual1} (3).  Hence the lemma.
\end{proof}
Now we start to modify Proposition \ref{ft}.

Let $\mf^S$ ($\mf^T$, resp.) be a $S$-flat ($T$-flat, resp.) family of sheaves in $\widetilde{U}(d,0)$ ($\widetilde{U}(r,0)$, resp.), and moreover for a generic $(s,t)\in S\times T$, the intersection of supports of $\mf^S_s$ and $\mf^T_t$ is of dimension 0.  Let $\mg^T$ be the Fourier transform of $\mf^T$.  Let $(\mf^S)^D$ be the dual of $\mf^S$ as defined in Lemma \ref{dual1} and let $(\mg^S)^{\vee}$ be the analog to Lemma \ref{dual2}.  Then by Lemma \ref{dual1} and a direct observation, $(\mf^S)^D$ is a $S$-flat family of sheaves in $\widetilde{U}(d,0)$.  By Lemma \ref{dual2}, $(\mg^S)^{\vee}$ is $S$-flat.

Denote by $\z^r_d(r)^D|_S$ ($\z^d_r(d)|_T$, resp.) the pullback of $\z^r_d(r)$ ($\z^d_r(d)$, resp.) on $\widetilde{U}(d,0)$ ($\widetilde{U}(r,0)$, resp.) via the classifying map $S\ra \widetilde{U}(d,0)$ ($T\ra \widetilde{U}(r,0)$) induced by $(\mf^S)^D$ ($\mf^T$, resp.), and denote by $\lambda_d(r)^{\vee}|_S$ ($\lambda_r(d)|_T$, resp.) the pullback of $\lambda_d(r)$ ($\lambda_r(d)$, resp.) on $W(d,0,d)^g$ ($W(r,0,r)^g$, resp.) via the classifying map $S\ra W(d,0,d)^g$ ($T\ra W(r,0,r)^g$) induced by $(\mg^S)^{\vee}$ ($\mg^T$, resp.).  Then by Corollary \ref{evo1} and Lemma \ref{evo2}, we have $\z^r_d(r)^D|_S\cong\z^r_d(r)|_S$ and $\lambda_d(r)^{\vee}|_S\cong\lambda_d(r)|_S$.

Define
$$\widehat{\rd^1}:=\big\{(s,t)\in S\times T\big| H^0((\mf^S)^D_s\otimes \mg^T_t)\neq 0\big\};$$
$$\widehat{\rd^2}:=\big\{(s,t)\in S\times T\big|H^0((\mg^S)^{\vee}_s\otimes\mf^T_t)\neq 0\big\}.$$
Then according to Theorem 2.1 in \cite{nila}, $\widehat{\rd^1}$ ($\widehat{\rd^2}$, resp.) is a divisor of line bundle $\z^r_d(r)|_S\boxtimes \lambda_r(d)|_T$ ($\lambda_d(r)|_S\boxtimes \z^d_r(d)|_T$, resp.), which induces a map $\zeta_{\widehat{\rd^1}}$ ($\zeta_{\widehat{\rd^2}}$, resp.) as follows.
\[H^0(S,\z^r_d(r)|_S)^{\vee}\xlongrightarrow{\zeta_{\widehat{\rd^1}}} H^0(T,\lambda_r(d)|_T);\]
\[H^0(S,\lambda_d(r)|_S)^{\vee}\xlongrightarrow{\zeta_{\widehat{\rd^2}}} H^0(T,\z^d_r(d)|_T).\]
\begin{prop}\label{mft}$\z^r_d(r)|_S\boxtimes \lambda_r(d)|_T\cong\lambda_d(r)|_S\boxtimes \z^d_r(d)|_T$ and $\widehat{\rd^1}=\widehat{\rd^2}$.  In particular, we have the following commutative diagram
\begin{equation}\xymatrix{H^0(S,\z^r_d(r)|_S)^{\vee}\ar[r]^{\zeta_{\widehat{\rd^1}}}\ar[d]_{\cong}&H^0(T,\lambda_r(d)|_T)\ar[d]^{\cong}\\
H^0(S,\lambda_d(r)|_S)^{\vee}\ar[r]_{\zeta_{\widehat{\rd^2}}} &H^0(T,\z^d_r(d)|_T). }
\end{equation}
\end{prop}
\begin{proof}The proof is analogous to Proposition \ref{ft}.  But we replace $\widetilde{\mg}$ by $\widetilde{\mf}$ defined as follows.
$$\widetilde{\mf}:=((id_{S\times T}\times p)^*\alpha_S^*((\mf^S)^{D}\otimes \tau_S^*\mo_{\p^2}(-1)))\otimes ((id_{S\times T}\times q)^*\alpha_T^*(\mf^T\otimes \tau_{T}^*\mo_{\p^2}(2))).$$
Then by Lemma \ref{tor} (3) and Lemma \ref{prof} we have 
\[R^i(id_{S\times T}\times p)_{*}\widetilde{\mf}=0,~\forall~i>0,~(id_{S\times T}\times p)_{*}\widetilde{\mf}\cong \alpha_S^{*}(\mf^S)^D\otimes\alpha_T^{*}\mg^T,\]
and 
\[R^i(id_{S\times T}\times q)_{*}\widetilde{\mf}=0,~\forall~i>0,~(id_{S\times T}\times q)_{*}\widetilde{\mf}\cong \alpha_S^{*}(\mg^S)^{\vee}\otimes\alpha_T^{*}\mf^T.\]
Then by following the rest argument of the proof of Proposition \ref{ft}, the proposition is proved. 
\end{proof}

\begin{proof}[Proof of Theorem \ref{App1}]By Proposition \ref{mft},  $\Phi^{*}\lambda_d(r)\cong \lambda_{d}(r)=\z^r_{d}(r),\forall ~d,r.$ 

Since both $\md$ and $W(d,0,d)$ are normal and irreducible, $\Phi_{*}\mo_{\md^g}\cong \mo_{W(d,0,d)^g}$.  Therefore
$\Phi_{*}\z_{d}^r(r)\cong \lambda_d( r),\forall~d,r$.  By Lemma \ref{ftg1} we have
\[\begin{array}{c}H^0(\md,\z^r_{dH}(r))\xrightarrow{\cong} H^0(\md^g,\z^r_{dH}(r))\xrightarrow{\cong}H^0(W(d,0,d)^g,\Phi_{*}\z_{dH}^r(r))\\
\xrightarrow{\cong} H^0(W(d,0,d)^g,\lambda_d(H^{\otimes r}))\xrightarrow{\cong} H^0(W(d,0,d),\lambda_d(H^{\otimes r}))\end{array}.\]
The theorem is proved.
\end{proof}

Yao YUAN \\
Yau Mathematical Sciences Center, Tsinghua University, \\
Beijing 100084, China\\
E-mail: yyuan@mail.tsinghua.edu.cn; yyuan@math.tsinghua.edu.cn.

\end{document}